\documentclass[a4paper,11pt]{amsart}
\usepackage[left=2.7cm,right=2.7cm,top=3.5cm,bottom=3cm]{geometry}

\usepackage{amsthm,amssymb,amsmath,amsfonts,mathrsfs,amscd,dsfont}
\usepackage[all]{xy}
\usepackage{latexsym}
\usepackage{longtable}
\usepackage{url}

\newfont{\cyr}{wncyr10 scaled 1100}
\newfont{\cyrr}{wncyr9 scaled 1000}

\theoremstyle{plain}
\newtheorem{theorem}{Theorem}[section]
\newtheorem{proposition}[theorem]{Proposition}
\newtheorem{lemma}[theorem]{Lemma}
\newtheorem{corollary}[theorem]{Corollary}

\theoremstyle{definition}
\newtheorem{conjecture}[theorem]{Conjecture}
\newtheorem{definition}[theorem]{Definition}
\newtheorem{assumption}[theorem]{Assumption}
\newtheorem*{mainconjecture}{Main Conjecture}

\theoremstyle{remark}
\newtheorem{remark}[theorem]{Remark}

\newcommand{\Q}{\mathds Q}

\newcommand{\Z}{\mathds Z}

\newcommand{\C}{\mathds C}
\newcommand{\D}{\mathds D}

\DeclareMathOperator{\Frob}{Frob}

\DeclareMathOperator{\Hom}{Hom}

\DeclareMathOperator{\Gal}{Gal}
\DeclareMathOperator{\GL}{GL}

\DeclareMathOperator{\Sel}{Sel}

\DeclareMathOperator{\KS}{\mathbf{KS}}
\DeclareMathOperator{\Ind}{Ind}
\DeclareMathOperator{\Fil}{Fil}

\newcommand{\res}{\mathrm{res}}
\newcommand{\cores}{\mathrm{cores}}
\newcommand{\cor}{\mathrm{cores}}

\newcommand{\tr}{\mathrm{tr}}
\newcommand{\ord}{\mathrm{ord}}

\usepackage[usenames]{color}
\definecolor{Indigo}{rgb}{0.2,0.1,0.7}
\definecolor{Violet}{rgb}{0.5,0.1,0.7}
\definecolor{White}{rgb}{1,1,1}
\definecolor{Green}{rgb}{0.1,0.9,0.2}

\newcommand{\longmono}{\mbox{\;$\lhook\joinrel\longrightarrow$\;}}

\newcommand{\longepi}{\mbox{\;$\relbar\joinrel\twoheadrightarrow$\;}}

\newcommand{\dirlim}{\mathop{\varinjlim}\limits}
\newcommand{\invlim}{\mathop{\varprojlim}\limits}

\newcommand{\p}{\mathfrak p}

\newcommand{\m}{\mathfrak m}

\newcommand{\cO}{{\mathcal O}}

\newcommand{\HH}{{\mathcal H}}

\newcommand{\T}{\mathbf{T}}

\include{thebibliography}

\begin{document}

\title[Kolyvagin systems of generalized Heegner cycles]{Kolyvagin systems and Iwasawa theory\\of generalized Heegner cycles}


\author{Matteo Longo and Stefano Vigni}

\thanks{The authors are partially supported by PRIN 2010--11 ``Arithmetic Algebraic Geometry and Number Theory''. The first author is also partially supported by PRAT 2013 ``Arithmetic of Varieties over Number Fields''; the second author is also partially supported by PRA 2013 ``Geometria Algebrica e Teoria dei Numeri''.}

\begin{abstract}
Iwasawa theory of Heegner points on abelian varieties of $\GL_2$ type has been studied by, among others, Mazur, Perrin-Riou, Bertolini and Howard. The purpose of this paper is to describe extensions of some of their results in which abelian varieties are replaced by the Galois cohomology of Deligne's $p$-adic representation attached to a modular form of even weight $>2$. In this setting, the role of Heegner points is played by higher-dimensional Heegner-type cycles that have been recently defined by Bertolini, Darmon and Prasanna. Our results should be compared with those obtained, via deformation-theoretic techniques, by Fouquet in the context of Hida families of modular forms.
\end{abstract}

\address{Dipartimento di Matematica, Universit\`a di Padova, Via Trieste 63, 35121 Padova, Italy}
\email{mlongo@math.unipd.it}
\address{Dipartimento di Matematica, Universit\`a di Genova, Via Dodecaneso 35, 16146 Genova, Italy}
\email{vigni@dima.unige.it}



\maketitle


\section{Introduction}

Initiated by Mazur's paper \cite{Maz2}, Iwasawa theory of Heegner points on abelian varieties of $\GL_2$ type (most notably, elliptic curves) has been investigated by, among others, Perrin-Riou (\cite{PR}), Bertolini (\cite{Ber1}, \cite{Ber2}) and Howard (\cite{Ho1}, \cite{Ho2}). A recurrent theme in all these works is the study of pro-$p$-Selmer groups, where $p$ is a prime number, in terms of Iwasawa modules built out of compatible families of Heegner points over the anticyclotomic $\Z_p$-extension of an imaginary quadratic field. In particular, several results on the structure of Selmer groups obtained by Kolyvagin by using his theory of Euler systems (\cite{Kol-Euler}) have been generalized to an Iwasawa-theoretic setting.

The goal of the present paper is to address similar questions in which abelian varieties are replaced by the Galois cohomology of Deligne's $p$-adic representation attached to a modular form $f$ of even weight $>2$. In this context, the role of Heegner points is played by generalized Heegner cycles defined by Bertolini, Darmon and Prasanna in \cite{BDP} or, rather, by a variant of them considered by Castella and Hsieh in \cite{CH}. As their name suggests, these cycles are a generalization of the Heegner cycles that were introduced by Nekov\'a\v{r} in \cite{Nek} in order to extend Kolyvagin's theory to Chow groups of Kuga--Sato varieties. 

Let $N\geq3$ be an integer, let $k\geq4$ be an even integer and let $f$ be a normalized newform of weight $k$ and level $\Gamma_0(N)$, whose $q$-expansion will be denoted by
\[ f(q)=\sum_{n\geq1}a_nq^n. \] 
Fix an imaginary quadratic field $K$ of discriminant coprime to $Np$ in which all the prime factors of $N$ split and let $p$ be a prime number not dividing $N$. Fix also embeddings $K\hookrightarrow\C$ and $\bar\Q\hookrightarrow\bar\Q_p$, where $\bar\Q$ and $\bar\Q_p$ denote algebraic closures of $\Q$ and $\Q_p$, respectively. Write $F$ for the number field generated over $\Q$ by the Fourier coefficients $a_n$ of $f$ and let $\cO_F$ be its ring of integers. Let $\p$ be a prime ideal of $\mathcal O_F$ above $p$ and denote by $V_{f,\p}$ the $p$-adic representation of $G_\Q:=\Gal(\bar\Q/\Q)$ attached to $f$ and $\p$ by Deligne (\cite{Del-Bourbaki}). If $F_\p$ is the completion of $F$ at $\p$ then $V_{f,\p}$ is an $F_\p$-vector space of dimension $2$ that comes equipped with a continuous action of $G_\Q$. Let $\cO_\p$ be the valuation ring of $F_\p$. In \S \ref{admissible-subsec} we introduce the notion of an \emph{admissible triple}: throughout this article we assume that $(f,K,\p)$ is such a triple. Here we content ourselves with pointing out that we insist that the prime $p$ be unramified in $F$ and split in $K$ and that $V_{f,\p}$ have large Galois image. Moreover, we require that $a_p\in\cO^\times_\p$, which is an ordinariness condition on $f$ at $p$.

Now let $K_\infty$ be the anticyclotomic $\Z_p$-extension of $K$ (i.e., the unique $\Z_p$-extension of $K$ that is Galois and dihedral over $\Q$), set $\Gamma_\infty:=\Gal(K_\infty/K)$ and form the Iwasawa algebra $\Lambda:=\cO_\p[\![\Gamma_\infty]\!]\simeq\cO_\p[[X]]$. As in \cite{Nek}, one can define a $G_\Q$-representation $T$ that is a free $\cO_\p$-module of rank $2$. Then $V:=T\otimes_{\cO_\p}F_\p$ is the $k/2$-twist of $V_{f,\p}$, and we set $A:=V/T$. Finally, denote by $H^1_f(K_\infty,A)$ the Bloch--Kato Selmer group  of $A$ over $K_\infty$ and by $\mathcal X_\infty$ the Pontryagin dual of $H^1_f(K_\infty,A)$, which is finitely generated over $\Lambda$ (Corollary \ref{coro-fin-gen}). Our main result describes, in particular, the $\Lambda$-module structure (up to pseudo-isomorphism) of $\mathcal X_\infty$. 

The key tool in our arguments is a certain $\Lambda$-submodule $\mathcal H_\infty$ of the pro-$p$ Bloch--Kato Selmer group $\hat H^1_{f}(K_\infty,T)$ of $f$ over $K_\infty$. We introduce this $\Lambda$-module, which is built in terms of the generalized Heegner cycles of Bertolini--Darmon--Prasanna as slightly modified by Castella--Hsieh, in Definition \ref{iwasawa-module-dfn}. Two features of the collection of cycles studied by Castella and Hsieh in \cite{CH} that we crucially exploit and that do not seem to be available for classical Heegner cycles are trace-compatibility and non-triviality along the $\Z_p$-extension $K_\infty/K$, the latter result representing a higher weight analogue of a well-known theorem of Cornut for Heegner points on rational elliptic curves (\cite{Co}).

Let $\iota:\Lambda\rightarrow\Lambda$ denote the involution induced by inversion in $\Gamma_\infty$ and if $X$ is a finitely generated torsion $\Lambda$-module then write $\mathrm{char}(X)$ for the characteristic ideal of $X$. Our main result, which is proved in \S \ref{main-proof-subsec}, is 

\begin{theorem} \label{main-thm}
Suppose that $(f,K,\p)$ is an admissible triple. Then there exist a finitely generated torsion $\Lambda$-module $M$ such that $\mathrm{char}(M)=\mathrm{char}(M)^\iota$ and a pseudo-isomorphism 
\[ \mathcal X_\infty\sim \Lambda\oplus M\oplus M. \]
Moreover, $\mathrm{char}(M)$ divides $\mathrm{char}\Big(\hat H^1_{f}(K_\infty,T)/\mathcal H_\infty\Big)$. 
\end{theorem}

We conjecturally refine the last part of Theorem \ref{main-thm} and propose in Conjecture \ref{main-conj} the following

\begin{mainconjecture}
$\mathrm{char}(M)=\mathrm{char}\Big(\hat H^1_{f}(K_\infty,T)/\mathcal H_\infty\Big)$.
\end{mainconjecture}

Our strategy for proving Theorem \ref{main-thm} is an extension to higher weights of arguments of Howard for elliptic curves and abelian varieties of $\GL_2$ type, i.e., for weight $2$ modular forms (\cite{Ho1}, \cite{Ho2}). In turn, Howard's work builds on and refines the theory of Kolyvagin systems developed by Mazur and Rubin (\cite{MR}). More precisely, the main contribution of our paper is the construction of a Kolyvagin system of generalized Heegner cycles and its application to the study of Iwasawa-theoretic questions for modular forms of even weight $\geq4$.

Finally, we remark that, by adopting a deformation-theoretic point of view and using (a generalization of) Howard's big Heegner points (\cite{Howard-Inv}), essentially the same results have been obtained by Fouquet in the context of Hida families of modular forms (\cite{Fouquet}). 

\section{Bloch--Kato Selmer groups in $\Z_p$-extensions}

Our goal in this section is to introduce the Selmer groups we shall be interested in and state a ``control theorem'' for them.

\subsection{Bloch--Kato Selmer groups} \label{sec-Bloch-Kato} 

We begin with a general discussion of Selmer groups of $p$-adic representations. 

For a number field $E$ and a crystalline $p$-adic representation $V$ of $G_E$ we introduce local conditions as follows. For primes $v\nmid p$ of $E$ we let $H^1_f(E_v,V)$ be the group of unramified cohomology classes, while for primes $v\,|\,p$ we define $H^1_f(E_v,V)$ to be the kernel of the map induced on cohomology by the natural map $V\rightarrow V\otimes_{\Q_p}B_\mathrm{cris}$, where $B_\mathrm{cris}$ is Fontaine's ring of crystalline periods. Moreover, for every place $v$ of $E$ we set
\[ H^1_s(E_v,V):=H^1(E_v,V)/H^1_f(E_v,V) \] 
and write
\[ \partial_v:H^1(E,V)\longrightarrow H^1_s(E_v,V) \] 
for the composition of the restriction $H^1(E,V)\rightarrow H^1(E_v,V)$ with the canonical projection. 

The \emph{Bloch--Kato Selmer group} of $V$ over $E$ (\cite[Sections 3 and 5]{BK}) is the group $H^1_f(E,V)$ that makes the sequence 
\[ 0\longrightarrow H^1_f(E,V)\longrightarrow H^1(E,V)\xrightarrow{\prod_v\partial_v}\prod_vH^1_s(E_v,V) \] 
exact. If $T$ is a $G_E$-stable lattice in $V$ then we set $A:=V/T$ and, for all integers $n\geq1$, let $A_{p^n}$ denote the $p^n$-torsion of $A$. There is a canonical isomorphism $A_{p^n}\simeq T/p^nT$. 

The projection $p:V\twoheadrightarrow A$ and the inclusion $i:T\hookrightarrow V$ induce 
maps 
\[ p: H^1(E,V)\longrightarrow H^1(E,A),\qquad i: H^1(E,T)\longrightarrow H^1(E,V); \] 
let us define $H^1_f(E,A):=p\bigl(H^1_f(E,V)\bigr)$ and $H^1_f(E,T):=i^{-1}\bigl(H^1_f(E,V)\bigr)$. Furthermore, the inclusion $i_n:A_{p^n}\hookrightarrow A$ and the projection $p_n:T\twoheadrightarrow T/p^nT$ induce maps 
\[ i_n:H^1(E,A_{p^n})\longrightarrow H^1(E,A),\qquad p_n:H^1(E,T)\longrightarrow H^1(E,T/p^nT); \]
we set $H^1_f(E,T/p^nT):=p_n\bigl( H^1_f(E,T)\bigr)$ and $H^1_f(E,A_{p^n}):=i_n^{-1}\bigl(H^1_f(E,A)\bigr)$. It can be checked that the isomorphisms $A_{p^n}\simeq T/p^nT$ induce isomorphisms $H^1_f(E,A_{p^n})\simeq H^1_f(E,T/p^nT)$ between Selmer groups.


Now let $M\in\bigl\{V, T, A, A_{p^n}, T/p^nT\bigr\}$. If $L/E$ is a finite extension of number fields then restriction and corestriction induce maps 
\[ \res_{L/E}:H^1_f(E,M)\longrightarrow H^1_f(L,M),\qquad\cor_{L/E}:H^1_f(L,M)\longrightarrow H^1_f(E,M). \] 
Finally, if $E$ is a number field and $\ell$ is a prime number then we set
\[ H^i_f(E_\ell,M):=\bigoplus_{\lambda|\ell}H^i_f(E_\lambda,M),\quad H^i_s(E_\ell,M):=\bigoplus_{\lambda|\ell}H^i_s(E_\lambda,M),\quad\partial_\ell:=\bigoplus_{\lambda|\ell}\partial_\lambda, \]
the direct sums being taken over the primes $\lambda$ of $E$ above $\ell$.

\subsection{Admissible triples} \label{admissible-subsec}

As in the introduction, let $f$ be a normalized newform of weight $k$ for $\Gamma_0(N)$ and write $\cO_F$ for the ring of integers of the number field $F$ generated by the Fourier coefficients of $f$. Let $\phi$ be Euler's function and let $\Xi$ be the set of prime numbers $p$ satisfying at least one of the following conditions:   
\begin{itemize} 
\item $p\,|\,6N(k-2)!\phi(N)c_f$; 
\item the image of the $p$-adic representation 
\[ \rho_{f,p}:G_\Q\longrightarrow\GL_2(\cO_F\otimes\Z_p) \] 
attached to $f$ by Deligne does not contain the set 
\[ \bigl\{g\in\GL_2(\cO_F\otimes\Z_p)\;\mid\;\det(g)\in(\Z_p^\times)^{k-1}\bigr\}. \]  
\end{itemize}

By \cite[Lemma 3.8]{LV}, the set $\Xi$ is finite. Let $\p$ be a maximal ideal of $\cO_F$ above $p$ and let $\cO_\p$ be the completion of $\cO_F$ at $\p$, whose field of fractions will be denoted by $F_\p$. Note that one recovers, up to isomorphism, $V_{f,\p}$ from $\rho_{f,p}$ by extending scalars to $\Q_p$ and projecting onto $\GL_2(F_\p)$. Finally, let $K$ be an imaginary quadratic field and write $h_K$ for its class number.

\begin{definition} \label{admissible-dfn}
The triple $(f,K,\p)$ is \emph{admissible} if 
\begin{enumerate} 
\item $p\notin\Xi\cup\{\text{$\ell$ prime}:\ell\,|\,h_K\}$;
\item  $p$ does not ramify in $F$;
\item $p$ splits in $K$;
\item \label{ass-cond} $a_p\in\cO_\p^\times$.
\end{enumerate} 
\end{definition}

\begin{remark}
Since $\Xi$ is finite, conditions (1) and (2) in Definition \ref{admissible-dfn} exclude only finitely many primes. On the other hand, the set of primes satisfying condition (3) has density $\frac{1}{2}$. Finally, in light of results of Serre on eigenvalues of Hecke operators (\cite[\S 7.2]{Serre-Cheb}), it seems reasonable to expect that condition (4), which is an ordinariness property of $f$ at $p$, holds for infinitely many $p$ (at least if $F\not=\Q$). In fact, questions of this sort appear to lie in the circle of ideas of the conjectures of Lang and Trotter on the distribution of traces of Frobenius automorphisms acting on elliptic curves (\cite{LT}) and of their extensions to higher weight modular forms (see, e.g., \cite{MM}, \cite{MMS}).
\end{remark}

Throughout this article we shall always work under the following

\begin{assumption} \label{ass}
The triple $(f,K,\p)$ is admissible.
\end{assumption}

Finally, we also assume, for simplicity, that $\cO_K^\times=\{\pm1\}$, i.e., that $K\not=\Q(\sqrt{-1})$ and $K\not=\Q(\sqrt{-3})$.

\subsection{The anticyclotomic $\Z_p$-extension of $K$} \label{anticyclotomic-subsec}

In general, for every integer $n\geq1$ we write $K[n]$ for the ring class field of $K$ of conductor $n$. The triple $(f,K,\p)$ is assumed to be admissible, hence $p\nmid h_K$; since $p$ is unramified in $K$, we also have $p\nmid|\Gal(K[p]/K[1])|$. Moreover, $\Gal(K[p^{m+1}]/K[p])\simeq\Z/p^m\Z$ for all $m\geq1$. It follows that for all $m$ there is a splitting 
\[ \Gal(K[p^{m+1}]/K)\simeq \Gamma_m\times\Delta \] 
with $\Gamma_m\simeq\Z/p^m\Z$ and $\Delta\simeq\Gal(K[p]/K)$ of order prime to $p$. For every $m\geq1$ define $K_m$ as the subfield of $K[p^{m+1}]$ that is fixed by $\Delta$, so that 
\[ \Gal(K_m/K)=\Gamma_m\simeq\Z/p^m\Z. \]
The field $K_\infty:=\cup_{m\geq1}K_m$ is the anticyclotomic $\Z_p$-extension of $K$; equivalently, it is the $\Z_p$-extension of $K$ that is (generalized) dihedral over $\Q$. Set
\[ \Gamma_\infty:=\varprojlim_m \Gamma_m=\Gal(K_\infty/K)\simeq\Z_p. \] 
Furthermore, for every $m\geq1$ set $\Lambda_m:=\cO_\p[\Gamma_m]$ and define
\[ \Lambda:=\varprojlim_m\Lambda_m=\cO_\p[\![\Gamma_\infty]\!]. \]
Here the inverse limit is taken with respect to the maps induced by the natural projections
$\Gamma_{m+1}\rightarrow \Gamma_m$. For all $m\geq1$ fix a generator $\gamma_m$ of $\Gamma_m$ in such a way that ${\gamma_{m+1}|}_{K_m}=\gamma_m$; then $\gamma_\infty:=(\gamma_1,\dots,\gamma_m,\dots)$ is a topological generator of $\Gamma_\infty$. It is well known that the map
\[ \Lambda\overset\simeq\longrightarrow\cO_\p[[X]],\qquad\gamma_\infty\longmapsto1+X \]
is an isomorphism of $\cO_\p$-algebras (see, e.g., \cite[Proposition 5.3.5]{NSW}).

For an abelian pro-$p$ group $M$ write $M^\vee:=\Hom_{\Z_p}^{\rm cont}(M,\Q_p/\Z_p)$ for its Pontryagin dual, equipped with the compact-open topology (here $\Hom_{\Z_p}^{\rm cont}$ denotes continuous homomorphisms of $\Z_p$-modules and $\Q_p/\Z_p$ is discrete). In the rest of the paper it will be convenient to use also the alternative definition $M^\vee:=\Hom_{\cO_\p}^{\rm cont}(M,F_\p/\cO_\p)$, where $\Hom_{\cO_\p}^{\rm cont}$ stands for continuous homomorphisms of $\cO_\p$-modules and $F_\p/\cO_\p$ is given the discrete topology. It turns out that the two definitions are equivalent, as there is a non-canonical isomorphism between $\Hom_{\cO_\p}^{\rm cont}(M,F_\p/\cO_\p)$ and $\Hom_{\Z_p}^{\rm cont}(M,\Q_p/\Z_p)$ that depends on the choice of a $\Z_p$-basis of $\cO_\p$. See, e.g., \cite[Lemma 2.4]{Bru} for details. 

Let $I_\infty:=(\gamma_\infty-1)$ be the augmentation ideal of $\Lambda$ and for every integer $n\geq0$ consider the ideal $I_n:=\bigl(\gamma_\infty^{p^n}-1\bigr)$ of $\Lambda$; in particular, $I_0=I_\infty$. Now let $M$ be a continuous $\Lambda$-module. As before, the dual $M^\vee$ inherits a structure of continuous $\Lambda$-module. Since $\gamma_\infty^{p^n}$ is a topological generator of $\Gamma_n$, for all $n\geq0$ there are equalities
\begin{equation} \label{app-pontr}
M[I_n]=M^{\Gamma_n},\qquad M_{\Gamma_n}=M/I_nM.
\end{equation} We also recall that if $M$ is compact then 
\begin{equation} \label{C2}  
{(M^\vee)}_{\Gamma_n}=M^\vee/I_nM^\vee\simeq\bigl(M^{\Gamma_n}\bigr)^\vee \end{equation}
for every $n\geq0$.

\subsection{The control theorem} 

We resume the notation and conventions that we introduced in \S \ref{admissible-subsec}; in particular, Assumption \ref{ass} is in force. Let $T$ be the $G_\Q$-representation considered by Nekov\'a\v{r} in \cite[Proposition 3.1]{Nek}, where it is denoted by $A_\p$. This is a free $\cO_\p$-module of rank $2$. The $G_\Q$-representation $V:=T\otimes_{\cO_\p}F_\p$ is then the $k/2$-twist of the representation $V_{f,\p}$. Finally, define the $G_\Q$-representation $A:=V/T$. As above, we shall write $A_{p^n}$ for the $p^n$-torsion submodule of $A$. Observe that
\begin{equation} \label{A-union-eq}
A=\bigcup_{n\geq1}A_{p^n}=\dirlim_nA_{p^n}
\end{equation}
where the direct limit is taken with respect to the natural inclusions $A_{p^n}\hookrightarrow A_{p^{n+1}}$.

\begin{lemma} \label{vanishing-0-lemma}
\begin{enumerate}
\item $H^0(K_m,A)=0$ for all $m\geq0$.
\item $H^0(K_\infty,A)=0$.
\item $H^0(K_\infty,A_{p^n})=0$ for all $n\geq0$.
\end{enumerate}
\end{lemma}

\begin{proof} Fix an integer $m\geq0$. The extension $K_m/\Q$ is solvable, so $H^0(K_m,A_{p^n})=0$ for all $n\geq0$ by \cite[Lemma 3.10, (2)]{LV}. It follows from \eqref{A-union-eq} that
\[ H^0(K_m,A)=H^0\Big(K_m,\dirlim_nA_{p^n}\Big)=\dirlim_nH^0(K_m,A_{p^n})=0, \]
which proves part (1). Since $H^0(K_\infty,A)=\sideset{}{_m}\dirlim H^0(K_m,A)$, part (2) follows as well. Finally, for all $n\geq0$ one has $H^0(K_\infty,A_{p^n})=\sideset{}{_m}\varinjlim H^0(K_m,A_{p^n})$, which implies part (3). \end{proof}

Define the discrete $\Lambda$-module 
\[ H^1_f(K_\infty,{A}):=\dirlim_mH^1_f(K_m,{A}), \] 
the injective limit being taken with respect to the restriction maps. 
The following version of Mazur's control theorem was proved in \cite{Ochiai}. 

\begin{theorem}[Ochiai] \label{CT}
For every $m\geq0$ the canonical restriction map 
\[ \res_{K_\infty/K_m}:H^1_f(K_m,{A})\longmono H^1_f(K_\infty,{A})^{\Gal(K_\infty/K_m)}\]
is injective and has finite cokernel of order bounded independently of $m$.  
\end{theorem}

\begin{proof} This is essentially \cite[Theorem 2.4]{Ochiai}. Fix an integer $m\geq0$. To begin with, the inflation-restriction exact sequence reads
\begin{equation} \label{inf-res-control-eq}
0\longrightarrow H^1\bigl(\Gal(K_\infty/K_m),H^0(K_\infty,A)\bigr)\longrightarrow H^1(K_m,A)\longrightarrow H^1(K_\infty,A)^{\Gal(K_\infty/K_m)}.
\end{equation}
On the other hand, $H^0(K_n,A)=0$ for all $n$ by Lemma \ref{vanishing-0-lemma}, hence
\[ H^0(K_\infty,A)=\varinjlim_nH^0(K_n,A)=0. \]
Thus sequence \eqref{inf-res-control-eq} gives an injection $H^1(K_m,A)\hookrightarrow H^1(K_\infty,A)^{\Gal(K_\infty/K_m)}$, which in turn restricts to an injection between Bloch--Kato Selmer groups. The bound on the cokernel follows from part (2) of \cite[Theorem 2.4]{Ochiai} together with the following observation. Since $p$ splits in $K$, the local extension $K_{\infty,v}/K_v$ is the cyclotomic $\Z_p$ extension for each prime $v$ of $K$ above $p$, hence conditions (ii) and (iii) in part (2) of \cite[Theorem 2.4]{Ochiai} are satisfied. Since these are the only conditions required for the proof of \cite[Theorem 2.4]{Ochiai}, the result follows exactly as in \cite{Ochiai}. \end{proof}
 
Let 
\[ \mathcal X_\infty:=\Hom_{\cO_\p}^{\rm cont}\bigl(H^1_f(K_\infty,{A}),F_\p/\cO_\p\bigr) \]
be the Pontryagin dual of $H^1_f(K_\infty,{A})$, equipped with its canonical structure of compact 
$\Lambda$-module. We want to prove that $\mathcal X_\infty$ is finitely generated over $\Lambda$. For every $m\geq0$ let 
\[ \mathcal X_m:=\Hom_{\mathcal O_\p}^{\rm cont}\bigl(H^1_f(K_m,{A}),F_\p/\mathcal O_\p\bigr) \]
be the Pontryagin dual of $H^1_f(K_m,A)$. Each $\mathcal X_m$ has a natural structure of $\Lambda_m$-module and there is a canonical isomorphism of $\Lambda$-modules $\mathcal X_\infty\simeq\sideset{}{_m}\invlim\mathcal X_m$. Note that, since the Galois representation $A$ is unramified outside $Np$, the $\cO_\p$-modules $\mathcal X_m$ are finitely generated.

\begin{corollary} \label{coro2.2} 
For every $m\geq0$ there is a canonical surjection ${(\mathcal X_\infty)}_{\Gal(K_\infty/K_m)}\twoheadrightarrow \mathcal X_m$. 
\end{corollary} 

\begin{proof} Fix an $m\geq0$. Thanks to \eqref{C2}, the injection of Theorem \ref{CT} gives, by duality, a surjection  
\begin{equation} \label{coro-CT}
\mathcal X_\infty/I_m\mathcal X_\infty\longepi\mathcal X_m
\end{equation} 
whose kernel is finite of order bounded independently of $m$. On the other hand, by \eqref{app-pontr}, the quotient $\mathcal X_\infty/I_m\mathcal X_\infty$ is canonically isomorphic to ${(\mathcal X_\infty)}_{\Gal(K_\infty/K_m)}$, and we are done. \end{proof} 

Now we can prove the main result of this subsection.

\begin{corollary} \label{coro-fin-gen}
The $\Lambda$-module $\mathcal X_\infty$ is finitely generated. 
\end{corollary}

\begin{proof} By choosing $m=0$ in \eqref{coro-CT}, we obtain a surjection $\mathcal X_\infty/I_\infty\mathcal X_\infty\twoheadrightarrow\mathcal X_0$, whose kernel has size bounded independently of $m$. Since $\mathcal X_0$ is a finitely generated $\mathcal O_\p$-module, the result follows from a topological version of Nakayama's lemma (\cite[Corollary 1.5]{Bru} or \cite[Corollary 5.2.18, (ii)]{NSW}).\end{proof}

\section{The $\Lambda$-adic Kolyvagin system argument} \label{euler-sec}

In this section we prove a slight generalization of the $\Lambda$-adic argument described in \cite[\S 2.2]{Ho1} that will lead to Theorem \ref{MR}.

\subsection{Selmer triples} \label{triples-subsec}

Fix a prime number $p$, a coefficient ring $R$ (i.e., a noetherian, complete local ring with finite residue field of characteristic $p$), an imaginary quadratic field $K$ and a finitely generated $R$-module $T$ equipped with a continuous linear action of $G_K:=\Gal(\bar\Q/K)$ that is unramified outside a finite set $\Sigma$ of primes of $K$. Let $\m$ be the maximal ideal of $R$ and put $\bar T:=T/\mathfrak mT$. For every prime $v$ of $K$ we write $K_v$ for the completion of $K$ at $v$ and choose a decomposition group $G_{K_v}\subset G_K$, whose inertia subgroup will be denoted by $I_{K_v}$. 

Let $\mathcal L_0=\mathcal L_0(T)$ be the set of degree $2$ primes of $K$ that do not belong to $\Sigma$ and do not divide $p$. We often identify a prime $\lambda\in\mathcal L_0$ with its residual characteristic $\ell$ and write $\lambda\,|\,\ell$. Consequently, we use indifferently the symbols $\lambda$ and $\ell$ to denote the dependence of an object on such a prime; for example, we write either $K_\lambda$ or $K_\ell$ for a given $\lambda\in\mathcal L_0$. As in \cite[Definition 1.2.1]{Ho1}, for every $\lambda\in \mathcal L_0$ let $I_\ell$ be the smallest ideal of $R$ containing $\ell+1$ and such that $\Frob_\lambda\in\Gal(K_v^{\mathrm{unr}}/K_v)$ acts trivially on $T/I_\ell T$. For an integer $k\geq1$ let $\mathcal L_k=\mathcal L_k(T)$ be the subset of $\ell\in\mathcal L_0$ such that $I_\ell\subset p^k\Z_p$, and for $\lambda\in\mathcal L_0$ set $G_\ell:=k_\lambda^\times/k_\ell^\times$, where $k_\lambda$ and $k_\ell$ are the residue fields at $\lambda$ and $\ell$, respectively. Finally, let $\mathcal N_k$ denote the set of square-free products of elements in $\mathcal L_k$ and for each $n\in\mathcal N_0$ define the ideal $I_n:=\sum_{\ell\mid n}I_\ell$ in $R$ and the group $G_n:=\otimes_{\ell\mid n}G_\ell$. By convention, 
$1\in\mathcal N_k$ for all $k$, $I_1=(0)$ and $G_1=\Z$. 

For each prime $v$ of $K$ such that $v\nmid p$ and $v\not\in\Sigma$ we write $H^1_{\mathrm s}(K_v,T)$ for the \emph{singular} part of $H^1(K_v,T)$, i.e., the quotient of $H^1(K_v,T)$ by the \emph{finite} part 
\[ H^1_\mathrm{f}(K_v,T):=H^1_{\mathrm{unr}}(K_v,T):=\ker\Big(H^1(K_v,T)\longrightarrow H^1(K_v^{\mathrm{unr}},T)\Big). \] 
For primes $v$ of residual characteristic different from $p$ we also let $K_v^{(p)}$ denote a maximal totally tamely ramified abelian $p$-extension of $K_v$ and define the \emph{transverse} subgroup as
\[ H^1_{\mathrm{tr}} (K_v,T):=\ker\Big(H^1(K_v,T)\longrightarrow H^1(K_v^{(p)},T)\Big). \]  
By \cite[Proposition 1.1.9]{Ho1}, if $|k_v^\times|\cdot T=0$ then $H^1_{\mathrm{tr}}(K_v,T)$ projects isomorphically onto $H^1_\mathrm{s}(K_v,T)$ and there is a canonical splitting 
\[ H^1(K_v,T)=H^1_\mathrm{f}(K_v,T)\oplus H^1_{\mathrm{tr}}(K_v,T). \] 
On the other hand, by \cite[Proposition 1.1.7]{Ho1} there are canonical isomorphisms 
\[ H^1_\mathrm{f}(K_v,T)\simeq T/(\mathrm{Fr}_v-1)T,\qquad H^1_\mathrm{s}(K_v,T)\otimes k_v^\times\simeq T^{\mathrm{Fr}_v=1}, \]
which give a finite-singular comparison isomorphism 
\[ \phi_v^\mathrm{fs}:H^1_\mathrm{f}(K_v,T)\simeq T\overset\simeq\longrightarrow H^1_\mathrm{s}(K_v,T)\otimes k_v^\times \] 
when $G_{K_v}$ acts trivially on $T$.

As in \cite[p. 1445]{Ho1}, for any $n\ell\in\mathcal N_0$ there is a finite-singular isomorphism 
\[ \phi_\ell^\mathrm{fs}:H^1_\mathrm{f}(K_\ell,T/I_{n\ell}T)\overset\simeq\longrightarrow H^1_\mathrm{s}(K_\ell,T/I_{n\ell }T)\otimes G_\ell. \]
Finally, for every prime $v$ of $K$ let 
\[\mathrm{loc}_v: H^1(K,T)\longrightarrow H^1(K_v,T)\] be the localization map. 

Now let $(T,\mathcal F,\mathcal L)$ be a \emph{Selmer triple}. We refer to \cite[Definition 3.1.3]{MR} and \cite[Definition 1.2.3]{Ho1} for details. In particular, we fix a finite set $\Sigma(\mathcal F)$ of primes of $K$ containing $\Sigma$, all the primes above $p$ ad all the archimedean primes. Then $\mathcal F$ is a \emph{Selmer structure} on $T$ in the sense of \cite[Definition 1.1.10]{Ho1} (cf. also \cite[Definition 2.1.1]{MR}), so it corresponds to the choice of a local condition (i.e., a subgroup) $H^1_\mathcal F(K_v,T)\subset H^1(K_v,T)$ at each prime $v\in\Sigma(\mathcal F)$. Moreover, $\mathcal L$ is a subset of $\mathcal L_0$ disjoint from 
$\Sigma(\mathcal F)$. Let $\mathcal N=\mathcal N(\mathcal L)$ be the set of squarefree products of primes in $\mathcal L$, with the convention that $1\in\mathcal N$. 

The \emph{dual} of $T$ is the $R$-module $T^*:=\Hom_R(T,R(1))$ equipped with the structure of $G_K$-module given by $(\sigma\cdot f)(x):=\sigma f(\sigma^{-1}t)$. We define a Selmer structure $\mathcal F^*$ on $T^*$ by taking $\Sigma(\mathcal F^*)=\Sigma(\mathcal F)$ and the orthogonal complements with respect to the local Tate pairings as local conditions (see \cite[\S 2.3]{MR} for details). In this way we obtain a Selmer triple $(T^*,\mathcal F^*,\mathcal L)$ 
such that $\mathcal F(n)^*=\mathcal F^*(n)$ for all $n\in \mathcal N$ (\cite[Example 2.3.2]{MR}). 

Denote by $\tau\in G_\Q$ a fixed extension of the non-trivial element of $\Gal(K/\Q)$. As in \cite[\S 1.3]{Ho1}, we require the pair $(T,\mathcal F )$ to satisfy the following conditions:
\begin{enumerate} 
\item[(H.0)] $T$ is a free $R$-module of rank $2$;
\item[(H.1)] $\bar T$ is an absolutely irreducible representation of $G_K$ over $R/\mathfrak m$; 
\item[(H.2)] there exists a Galois extension $F$ of $\Q$ containing $K$ such that $G_F$ acts trivially on $T$ and $H^1\bigl(F(\mu_{p^\infty})/K,\bar T\bigr)=0$;
\item[(H.3)] for every $v\in\Sigma(\mathcal F)$ the local condition $\mathcal F$ at $v$ is cartesian on the quotient category of $T$ (see \cite[Definitions 1.1.2 and 1.1.3]{Ho1} for details); 
\item[(H.4)] there is a perfect, symmetric, $R$-bilinear pairing $(\,,\,):T\times T\rightarrow R(1)$ satisfying 
$(s^\sigma,t^{\tau\sigma\tau^{-1}})=(s,t)^\sigma$ for every $s, t\in T$ and $\sigma\in G_K$ and such that for every place $v$ of $K$ the local condition $\mathcal F$ at $v$ is its own exact orthogonal complement under the induced local pairing
\[ {\langle\cdot,\cdot\rangle}_v:H^1(K_v,T)\times H^1(K_{\bar v},T)\longrightarrow R; \]
\item[(H.5)]
\begin{enumerate} 
\item the action of $G_K$ on $\bar T$ extends to an action of $G_\Q$ and the action of $\tau$ splits $\bar T=\bar T^+\oplus \bar T^-$ into one-dimensional $R/\mathfrak m$-subspaces, where $\bar T^\pm$ is defined as the $\pm$-eigenspace for the action of $\tau$ on $\bar T$;
\item the condition $\mathcal F$ propagated to $\bar T$ (cf. \cite[\S 1.1]{Ho1} for definitions) is stable under the action of $G_\Q$; 
\item if we assume (H.4) to hold then the residual pairing 
\[ (\,,\,):\bar T\times\bar T\longrightarrow(R/\mathfrak m)(1) \] 
obtained from $(\,,\,)$ satisfies $(s^\tau,t^\tau)=(s,t)^\tau$ for all $s,t\in\bar T$. 
\end{enumerate}
\end{enumerate} 
See \cite[\S 1.3]{Ho1} for a comparison between these conditions and those, similar, in \cite[\S3.5]{MR}. 
One then defines the Selmer group attached to the Selmer structure $(T,\mathcal F)$ as 
\begin{equation} \label{selmer-group-F-eq}
H^1_\mathcal F(K,T):=\ker\!\bigg(H^1(K,T)\longrightarrow\bigoplus _{v}H^1(K_v,T)/H^1_{\mathcal F}(K_v,T)\bigg). 
\end{equation}

\subsection{Kolyvagin systems} \label{kolyvagin-subsec}

Fix a Selmer triple $(T,\mathcal F,\mathcal L)$ such that the pair $(T,\mathcal F)$ satisfies the assumptions above. Given $c\in \mathcal N$, we introduce a new Selmer triple $\bigl(T,\mathcal F(c),\mathcal L(c)\bigr)$ by defining 
\[ \Sigma\bigl(\mathcal F(c)\bigr):=\Sigma(\mathcal F)\cap\bigl\{v: v\,|\,c\bigr\},\qquad\mathcal L(c):=\bigl\{v\in\mathcal L:v\nmid c\bigr\} \] 
and taking 
\[ H^1_{\mathcal F(c)}(K_v,T):=\begin{cases}H^1_{\mathcal F}(K_v,T)&\text{if $v\nmid c$},\\[3mm] H^1_\mathrm{tr}(K_v,T)&\text{if $v\,|\,c$}.\end{cases}\]
For every product $n\ell\in\mathcal N_0$ there is a diagram
\begin{equation} \label{diagram}
\xymatrix@C=29pt{&&H^1_{\mathcal F(n\ell)}(K,T/I_{n\ell }T)\otimes G_{n\ell}\ar[d]^-{\mathrm{loc}_\ell}\\
H^1_{\mathcal F(n)}(K,T/I_nT)\otimes G_n\ar[r]^-{\mathrm{loc}_\ell}& H^1_\mathrm{f}(K_\ell,T/I_{n\ell}T)\otimes G_n\ar[r]^-{\phi_\ell^\mathrm{fs}\otimes1}&
H^1_\mathrm{s}(K_\ell,T/I_{n\ell}T)\otimes G_{n\ell}}
\end{equation}
whose row is exact. Let $(T,\mathcal F,\mathcal L)$ denote a Selmer triple. 

\begin{definition}\label{def-Kol-sys}
A \emph{Kolyvagin system} for $(T,\mathcal F,\mathcal L)$ is a collection $\kappa={\{\kappa_n\}}_{n\in \mathcal N(\mathcal L)}$ of classes 
\[ \kappa_n\in H^1_{\mathcal F(n)}(K,T/I_nT)\otimes G_n \] 
such that for every $n\ell\in\mathcal N(\mathcal L)$ the images of $\kappa_n$ and $\kappa_{n\ell}$ under the maps $(\phi_\ell^\mathrm{fs}\otimes1)\circ\mathrm{loc}_\ell$ and $\mathrm{loc}_\ell$ in \eqref{diagram} agree.  
\end{definition} 

The set of all Kolyvagin systems for $(T,\mathcal F,\mathcal L)$ is naturally endowed with an $R$-module structure. This $R$-module will be denoted by $\KS(T,\mathcal F,\mathcal L)$; we refer the reader to \cite[Remark 1.2.4]{Ho1} and \cite[Remark 3.1.4]{MR} for the functorial properties enjoyed by it.

\subsection{$\Lambda$-adic representations} \label{sec4.3} 

Let $\cO$ be the valuation ring of a finite extension $F$ of $\Q_p$, with maximal ideal $\m=(\pi)$ and residue field $k$. Let $\Lambda=\cO[\![\Gamma_\infty]\!]$ be the Iwasawa algebra with coefficients in $\cO$ of the anticyclotomic $\Z_p$-extension of $K$. Let $T$ be a free $\cO$-module of rank $2$ equipped with a continuous linear action of $G_\Q$ that is unramified outside a finite set of primes $\Sigma$ of $\Q$. Set $V:=T\otimes_{\mathcal O}F$ and $A:=V/T$. 

In addition to conditions (H.0)-(H.4) with $R=\cO$, we impose on $T$ the following set of assumptions, which are all verified when $T$, as in the case that concerns us, arises from a modular form.

\begin{assumption} \label{ass1} \begin{enumerate}
\item The $p$-adic representation $V$ is crystalline and for every prime $v$ of $K$ above $p$ its restriction to $G_{\Q_p}$ is equipped with a filtration of $G_{\Q_p}$-modules 
\[ 0\longrightarrow \Fil^+_v(T)\longrightarrow T \longrightarrow \Fil^-_v(T)\longrightarrow0 \]
where $\Fil^\pm_v(T)$ are both free of rank $1$ over $\cO$ and inertia acts trivially on $\Fil^-_v(T)$. 
\item The pairing in (H.4) gives rise to a pairing 
\[ (\,,):T\times A\longrightarrow (F/\cO)(1) \]
still denoted by the same symbol, and we require that $\Fil_v^+(T)$ and $\Fil_v^+(A)$ be exact annihilators of each other under $(\,,)$.
\item The groups $H^0\bigl(K_{\infty,v},\Fil^-_v(A)\bigr)$ and $H^0\bigl(K_{v},\Fil^-_v(A)\bigr)$ are finite for all primes $v\,|\,p$, where $K_{\infty,v}$ denotes the completion of $K_\infty$ at the unique prime $v$ above $p$. 
\end{enumerate}
\end{assumption} 
With notation as in part (1) of Assumption \ref{ass1}, let us define
\[ \Fil^\pm_v(V):=\Fil^\pm_v(T)\otimes_\cO F,\qquad\Fil^\pm_v(A):=\Fil^\pm_v(V)/\Fil^\pm_v(T). \]
For any left $G_K$-module $M$ and any finite extension $L/K$ we denote by $\mathrm{Ind}_{L/K}(M)$ the induced module from $K$ to $L$ of $M$, whose elements are functions $f:G_K\rightarrow M$ satisfying $f(\sigma x)=\sigma f(x)$ for all $x\in G_K$ and $\sigma\in\Gal(\bar\Q/L)$. This is endowed with right and left actions of $G_K$ and $\Gal(L/K)$ defined, respectively, by 
$(f^\sigma)(x)=f(x\sigma)$ and $(\gamma\cdot f)(x):={\tilde\gamma}f(\tilde\gamma^{-1}x)$ 
for all $\sigma\in G_K$ and $\gamma\in \Gal(L/K)$, where $\tilde\gamma\in G_K$ is any lift of $\gamma$. There are corestriction maps
\[ \cor_{m}:\Ind_{K_m/K}(M)\longrightarrow \Ind_{K_{m-1}/K}(M),\qquad f\longmapsto \sum_{\gamma\in\Gal(K_m/K_{m-1})}\gamma\cdot f \]
and restriction maps $\res_{m}:\Ind_{K_m/K}(M)\rightarrow \Ind_{K_{m+1}/K}(M)$ taking $f$ to $f$ itself. Define 
\[ \mathbf{T}:=T\otimes _\mathcal O\Lambda\simeq\invlim_m\mathrm{Ind}_{K_m/K}(T),\qquad\mathbf{A}:=\dirlim_m\Ind_{K_m/K}(A), \]
where the inverse and direct limits are taken with respect to corestrictions and restrictions, respectively. With notation as in Assumption \ref{ass1}, we set
\[ \Fil^\pm_v(\T):=\Fil^\pm_v(T)\otimes_\cO\Lambda\simeq\invlim_m\mathrm{Ind}_{K_m/K}\bigl(\Fil^\pm_v(T)\bigr),\quad\Fil^\pm_v(\mathbf{A}):=\dirlim_m\Ind_{K_m/K}\bigl(\Fil^\pm_v(A)\bigr). \]
We know that for any $n$ prime to $p$ there is an isomorphism
\begin{equation} \label{hat-H-T-eq}
H^1(K[n],\mathbf{T})\simeq \invlim_mH^1(K_m[n],T)=:\hat H^1(K_\infty[n],T) 
\end{equation}
where the limit on the right is computed with respect to the corestriction maps. To show this, one can use the arguments in the proof of \cite[Proposition II.1.1]{Colmez}, as in \cite[Lemma 5.3.1]{MR}. 

As in \cite[Proposition 2.2.4]{Ho1}, one obtains from the pairing $(\,,):T\times A\rightarrow(F/\cO)(1)$ in Assumption \ref{ass1} another pairing 
\begin{equation} \label{pairing2}
{(\,,)}_\infty:\mathbf{T}\times\mathbf{A}\longrightarrow (F/\cO)(1)
\end{equation} 
such that ${(\lambda x,y)}_\infty={(x,\lambda^* y)}_\infty$ for all $x\in\mathbf{T}$, $y\in\mathbf{A}$, $\lambda\in\Lambda$, where $\lambda\mapsto\lambda^*$ is the $\Z_p$-linear involution of $\Lambda$ defined as $\gamma\mapsto \gamma^{-1}$ on group-like elements. As a consequence of (2) in Assumption \ref{ass1}, $\Fil_v^+(\mathbf{T})$ and $\Fil_v^+(\mathbf{A})$ are exact annihilators of each other under ${(\,,)}_\infty$. 

For a number field $E$ and a prime $v\,|\,p$ of $E$, the \emph{local Greenberg condition} $H^1_\ord(E_v,T)$ at $v$ is the kernel of the map from $H^1(E_v,T)$ to $H^1(E_v,\Fil^-_v(T))$; this is also the image of $H^1(E_v,\Fil^+_v(T))$ inside $H^1(E_v,T)$. Then we define \emph{Greenberg's Selmer group} as
{\small{\[ \Sel_{\mathrm{Gr}}(T/E):=\ker\Bigg(H^1(E,T)\longrightarrow \bigoplus_{v\nmid p}H^1(
E_v,T)/H^1_\mathrm{unr}(E_v,T)\oplus\bigoplus_{v\mid p}H^1(E_v,T)/H^1_\ord(E_v,T)\Bigg). \]}}

We impose local conditions similar to Greenberg's on the big Galois representation $\mathbf{T}$. More precisely, we define a Selmer structure $\mathcal F_\Lambda$ on $\mathbf{T}$ by taking the unramified subgroup of $H^1(K_v,\mathbf{T})$ at primes $v\nmid p$ and 
\[ H^1_\ord(K_v,\T):=\mathrm{im}\Big(H^1\bigl(K_v,\Fil^+_v(\mathbf{T})\bigr)\longrightarrow H^1(K_v,\mathbf{T})\Big) \] 
at primes $v\,|\,p$. Then, as in \eqref{selmer-group-F-eq}, we introduce the corresponding Selmer group $H^1_{\mathcal F_\Lambda}(K,\T)$. 

\subsection{Structure theorems} \label{Lambda-adic-subsec}

Fix a height one prime ideal $\mathfrak P\neq p\Lambda$ of $\Lambda$, write $S_\mathfrak P$ for the integral closure of $\Lambda/\mathfrak P$ in its quotient field $\Phi_\mathfrak P$ and define $T_\mathfrak P:=T\otimes_\cO S_\mathfrak P$. Moreover, set $V_\mathfrak P:=T_\mathfrak P\otimes_{S_\mathfrak P}\Phi_\mathfrak P$. The pairing $(\,,):T\times T\rightarrow \mathcal O(1)$ gives rise to a pairing $e_\mathfrak P:T_\mathfrak P\times T_\mathfrak P\rightarrow S_\mathfrak P$ satisfying $e_\mathfrak P(s^\sigma,t^{\tau\sigma\tau^{-1}})=e_\mathfrak P(s,t)^\sigma$ for all $s,t\in T_\mathfrak P$ and all $\sigma\in G_K$ and such that $\Fil^+(T_\mathfrak P)$ is its own exact orthogonal complement. 

For any prime $v$ of $K$ above $p$ we define 
\[ \mathrm{Fil}_v^+(T_\mathfrak P):=\mathrm{Fil}_v^+(T)\otimes_\cO S_\mathfrak P,\qquad\mathrm{Fil}_v^+(V_\mathfrak P):=\mathrm{Fil}_v^+(T_\mathfrak P)\otimes_{S_\mathfrak P}\Phi_\mathfrak P. \] 
As above, the \emph{Greenberg condition} at $v\,|\,p$ is given by
\[ H^1_\ord(K_v,V_\mathfrak P):=\mathrm{im}\Big(H^1\bigl(K_v,\Fil^+(V_\mathfrak P)\bigr)\longrightarrow H^1(K_v,V_\mathfrak P)\Big). \]
One then considers the local condition at a prime $v$ of $K$ defined as 
\[ H^1_{\mathcal F_\mathfrak P}(K_v,V_\mathfrak V):=\begin{cases}H^1_\ord (K_v,V_\mathfrak P) & \text{if $v\,|\,p$,}\\[3mm]H^1_\mathrm{unr}(K_v,V_\mathfrak P) & \text{if $v\nmid p$}.
\end{cases} \]
By an abuse of notation, we also denote by $\mathcal F_{\mathfrak P}$ the local conditions obtained on $T_\mathfrak P$ and $A_\mathfrak P:=V_\mathfrak P/T_\mathfrak P$ by propagation. 

\begin{proposition} \label{prop-Howard1}
Fix an integer $s\geq1$ and a set of primes $\mathcal L\supset\mathcal L_s(T_\mathfrak P)$, and suppose that the Selmer triple $(T_\mathfrak P, \mathcal F_{\mathfrak P}, \mathcal L)$ admits a non-trivial Kolyvagin system $\kappa$. 
\begin{enumerate}
\item $H^1_{\mathcal F_\mathfrak P}(K,T_\mathfrak P)$ is a free $S_\mathfrak P$-module of rank $1$.
\item $H^1_{\mathcal F_\mathfrak P}(K,A_\mathfrak P)\simeq\Phi_\mathfrak P/S_\mathfrak P\oplus M_\mathfrak P\oplus M_\mathfrak P$ where $M_\mathfrak P$ is a finite $S_\mathfrak P$-module with 
\[ \mathrm{length}(M_\mathfrak P)\leq \mathrm{length}\Big(H^1_{\mathcal F_\mathfrak P}(K,T_\mathfrak P)/(S_\mathfrak P\cdot\kappa_1)\Big). \]
\end{enumerate}
\end{proposition}

\begin{proof} As in the proof of \cite[Proposition 2.1.3]{Ho1}, we can apply \cite[Theorem 1.6.1]{Ho1} once we show that $T_\mathfrak P$ satisfies (H.1)--(H.5). The verification of this property goes as in the proof of \cite[Proposition 2.1.3]{Ho1}; we just need to replace $E[p]$ and $E(K_\infty)[p]$ in \emph{loc. cit.} with $A_p$ and $A_p(K_\infty)$, respectively, and observe that $A_p(K_\infty)=0$ by part (3) of Lemma \ref{vanishing-0-lemma}. \end{proof}

The involution $\iota$ of $\Lambda$ that is induced by $\gamma\mapsto\gamma^{-1}$ on group-like elements gives a map $S_\mathfrak P\rightarrow S_{\mathfrak P^\iota}$ that will be denoted by the same symbol. The map $\psi(t\otimes\alpha)=t^\tau\otimes\alpha^\iota$ induces a bijection $T_\mathfrak P\rightarrow T_{\mathfrak P^\iota}$, while the map $(x,y)\mapsto\mathrm{tr}\circ e_\mathfrak P(\psi^{-1}(x),y)$ (where $\tr$ is the trace form) defines a perfect, $G_K$-equivariant pairing $(\,,):T_{\mathfrak P^\iota}\times A_\mathfrak P\rightarrow \mu_{p^\infty}$ satisfying $(\lambda x,y)=(x,\lambda^\iota y)$. Dualizing $\mathbf{T}/\mathfrak P^\iota\mathbf{T}\rightarrow T_{\mathfrak P^\iota}$ and using the pairing above and the pairing ${(\,,)}_\infty$ in \eqref{pairing2}, we obtain a $G_K$-equivariant map $A_\mathfrak P\rightarrow \mathbf{A}[\mathfrak P]$ that gives a map 
\begin{equation} \label{map1}
H^1_{\mathcal F_\mathfrak P}(K,A_\mathfrak P)\longrightarrow 
H^1_{\mathcal F_\mathfrak P}(K,\mathbf{A})[\mathfrak P].\end{equation} 
Also, the projection $\mathbf{T}\twoheadrightarrow T_\mathfrak P$ induces a map 
\begin{equation}\label{map2}
H^1_{\mathcal F_\mathfrak P}(K,\mathbf{T})/\mathfrak PH^1_{\mathcal F_\mathfrak P}(K,\mathbf{T})\longrightarrow 
H^1_{\mathcal F_\mathfrak P}(K,T_\mathfrak P).\end{equation}

\begin{proposition} \label{prop-Howard2}
For all but finitely many prime ideals $\mathfrak P$ of $\Lambda$ the maps \eqref{map1} and \eqref{map2} have finite kernel and cokernels that are bounded by a constant depending only on $[S_\mathfrak P:\Lambda/\mathfrak P]$. 
\end{proposition}

\begin{proof} The statement follows from the analogues of \cite[Lemma 2.2.7 and Proposition 2.2.8]{Ho1}. The only difference with respect to \emph{loc. cit.} is in the case $v\,|\,p$ of \cite[Lemma 2.2.7]{Ho1}, for which one has to use condition (Fin) above. \end{proof} 

The next result is the counterpart of \cite[Theorem 2.2.10]{Ho1}.

\begin{theorem} \label{MR} 
Let $(\mathbf{T},\mathcal F_\Lambda,\mathcal L)$ be a Selmer triple satisfying {\rm{(H.1)--(H.5)}} and Assumption \ref{ass1}. Set $X:=H^1_{\mathcal F_\Lambda}(K,\mathbf{A})^\vee$. Suppose that for some $s\geq1$ the Selmer triple $(\mathbf{T},\mathcal F_\Lambda,\mathcal L_s)$ admits a Kolyvagin system $\kappa$ with $\kappa_1\neq 0$. Then
\begin{enumerate} 
\item $H^1_{\mathcal F_\Lambda}(K,\mathbf{T})$ is a torsion-free $\Lambda$-module of rank $1$;
\item there exist a torsion $\Lambda$-module $M$ such that $\mathrm{char}(M)=\mathrm{char}(M)^\iota$ and a pseudo-isomorphism 
\[ X\sim\Lambda\oplus M\oplus M; \]
\item $\mathrm{char}(M)$ divides $\mathrm{char}\bigl(H^1_{\mathcal F_\Lambda}(K,\mathbf{T})/\Lambda \kappa_1\bigr)$. 
\end{enumerate} 
\end{theorem}

\begin{proof} One first replaces \cite[Propositions 2.1.3 and 2.2.8]{Ho1} with Propositions \ref{prop-Howard1} and \ref{prop-Howard2}, respectively, and uses them, together with the vanishing $A_p(K_\infty)=0$ of part (3) of Lemma \ref{vanishing-0-lemma}, to show as in \cite[Lemma 2.2.9]{Ho1} that $H^1_{\mathcal F_\Lambda}(K,\mathbf{T})$ is torsion-free over $\Lambda$. Then one proceeds as in the proof of \cite[Theorem 2.2.10]{Ho1}. \end{proof}

\section{The $\Lambda$-adic Kolyvagin system of generalized Heegner cycles} \label{kolyvagin-section} 

The goal of this section is to explain how the generalized Heegner cycles of Bertolini, Darmon and Prasanna (\cite{BDP}) can be used to define a Kolyvagin system in the sense of the previous section. Actually, we will use a variant of these cycles introduced by Castella and Hsieh in \cite{CH}. Along the way we prove a technical result on universal norms (Lemma \ref{Univ-Norms}), much in the spirit of Perrin-Riou's relations in \cite[\S3.3]{PR}. To the best of our knowledge, these formulas are considered here for the first time.   

In this section, $T$ denotes the $\Z_p$-representation that is attached to the modular form $f$ satisfying Assumption \ref{ass}. As before, write $F$ for the number field generated over $\Q$ by the Fourier coefficients of $f$ and $\cO_F$ for the ring of integers of $F$. Fix a prime $\p$ of $F$ above $p$ and let $F_\p$ and $\cO_\p$ be the completions of $F$ and $\cO_F$ at $\p$, respectively. Finally, let $K$ be the imaginary quadratic field that was chosen in the introduction. 

\subsection{Generalized Heegner cycles} \label{generalized-heegner-subsec}

Let $\mathcal L_0(\mathbf{T})$ denote the set of degree $2$ primes of $K$ that do not divide $p$ and any other prime at which $\mathbf{T}$ is ramified. For any integer $k\geq0$ define $\mathcal L_k(\mathbf{T})$ to be the subset of $\ell\in\mathcal L_0(\mathbf{T})$ such that $I_\ell\subset p^k\Z_p$. Let $\mathcal L:=\mathcal L_1(\mathbf{T})$, so that $(\mathbf{T}, \mathcal F_\Lambda,\mathcal L)$ is a Selmer triple, and let $\mathcal N:=\mathcal N(\mathcal L)$. Finally, for every $\ell\in \mathcal N$ write $\lambda$ for the unique prime of $K$ above $\ell$ and fix a prime $\bar\lambda$ of $\bar\Q$ above $\ell$. 

Now let $n\in\mathcal N$ and let $K_m[n]$ denote the composite of $K_m$ and $K[n]$. Put $K_0:=K$ and define $K_\infty[n]:=\sideset{}{_m}\dirlim K_m[n]$. The prime $\bar\lambda$ determines a prime $\lambda_{np^m}\in K_{m}[n]$; we denote by ${K_m[n]}_{\bar\lambda}$ the completion of $K_m[n]$ at $\lambda_{np^m}$.

For each $m\geq0$ let $z_{np^m}\in H^1(K[np^m],{T})$ be the class defined in \cite[eq. (4.7)]{CH} when $\chi$ is the trivial character (thus, this is the class denoted by $z_{f,\chi,np^m}$ in \emph{loc. cit.} for $\chi$ equal to the trivial character). As is explained in \cite[Section 4]{CH}, these classes are built in terms of the generalized Heegner cycles introduced by Bertolini, Darmon and Prasanna in \cite{BDP}, to which the reader is referred for details. Actually, by \cite[Theorem 2.4]{LV}, for each $m$ the class $z_{np^m}$ belongs to the Selmer group $H^1_f(K[np^m],{T})$.
Define  
\begin{equation} \label{alpha-eq}
\alpha_m[n]:=\cores_{K[np^{m+1}]/K_m[n]}(z_{np^{m+1}})\in H^1_f(K_m[n],{T}). 
\end{equation}

\subsection{Perrin-Riou--type formulas} \label{PR-subsec}

We generalize the recursive formulas proved by Perrin-Riou in \cite[Proposition 3 and Lemma 4]{PR} and obtain an analogue of \cite[Corollary 5]{PR} in this context. To this aim, set $\mathcal G(n):=\Gal(K[n]/K)$, let $\sigma_\wp,\sigma_{\bar\wp}\in\mathcal G(n)$ be the Frobenius automorphisms at the primes $\wp$, $\bar\wp$ of $K$ above $p$ and put $\delta:=\bigl[K[np^m]:K_{m-1}[n]\bigr]$. Then $\delta$ does not depend on $m$. Now define the following elements of $\cO_\p[\mathcal G(n)]$: 
\[ \begin{split}
   \varrho&:=p^{k/2}-a_p\sigma_\wp+p^{(k-2)/2}\sigma_\wp^2,\\
   \bar\varrho&:=p^{k/2}-a_p\sigma_{\bar\wp}+p^{(k-2)/2}\sigma_{\bar\wp}^2,\\
   \Phi&:=\varrho\cdot\bar\varrho,\\
   \gamma_0&:=a_p-p^{(k-2)/2}(\sigma_\wp+\sigma_{\bar\wp}),\\
   \gamma_1&:=a_p\gamma_0-p^{k-2}\delta=a_p^2-p^{(k-2)/2}a_p(\sigma_\wp+\sigma_{\bar\wp})-p^{k-2}\delta.
   \end{split} \] 
Finally, define 
\[ \gamma_m:=a_p\gamma_{m-1}-p^{k-1}\gamma_{m-2} \] 
recursively for $m\geq2$. 

\begin{lemma} \label{lemma-Heegner}
For all $n\geq1$ one has the following relations: 
\begin{enumerate}
\item $\cor_{K_{m+1}[n]/K_m[n]}(\alpha_{m+1}[n])=a_p\alpha_m[n]-p^{k-2} \res_{K_m[n]/K_{m-1}[n]}(\alpha_{m-1}[n])$ $\forall$ $m\geq1$;  
\item $\cor_{K_m[{n\ell}]/K_m[{n}]}({\alpha}_m[n\ell])=a_\ell\alpha_m[n]$ for all primes $\ell\nmid c$ inert in $K$; 
\item $\mathrm{loc}_\ell({\alpha}_m[n\ell])=\res_{{K_m[n\ell]}_{\bar\lambda}/{K_m[n]}_{\bar\lambda}}\bigl(\mathrm{loc}_\ell(\alpha_m[n])^{\Frob_\ell}\bigr)$ for all primes $\ell\nmid cN$ inert in $K$; 
\item $\cor_{K_m[n]/K[n]}(\alpha_m[n])=\gamma_m z_n$ $\forall$ $m\geq0$. 
\end{enumerate}
\end{lemma}

\begin{proof} Parts (1) and (2) follow immediately from \cite[Proposition 4.4]{CH}, while part (3) is a consequence of \cite[Lemma 4.7]{CH}. Finally, part (4) follows inductively from (1) and (2). \end{proof}

\begin{lemma} \label{Univ-Norms}
For all $m\geq2$ there exist $q_m,r_m\in\cO_\p[\mathcal G(n)]$ with $q_{m+1}\equiv a_pq_{m}\pmod{\p}$, $q_2=1$ such that
\[ \gamma_m=q_m\Phi+p^{(m-1)k/2}r_m. \] 
\end{lemma}

\begin{proof} The elements $r_m$ are defined recursively by the formulas 
\begin{itemize}
\item $s_1:=0$ and $s_{m+1}:=\sigma_\wp s_m-\sigma_{\bar\wp}^{m-1}$ for all $m\geq2$,
\item $r_m:=\varrho s_n +\sigma_{\bar\wp}^{m-1}\gamma_0$ for all $m\geq1$. 
\end{itemize}
The lemma can then be proved by an inductive argument as in \cite[\S 3.3, Lemma 4]{PR}. \end{proof} 

\begin{corollary} \label{coro5PR}
If $M$ is a finitely generated $\cO_\p[\mathcal G(n)]$-module then
\[ \Phi M=\bigcap_{m\geq1}\gamma_mM. \]
\end{corollary}

\begin{proof} As in \cite[\S 3.3, Corollaire 5]{PR}, this follows from Lemma \ref{Univ-Norms} and the fact that, by (5) in Assumption \ref{ass}, $a_p$ is a $p$-adic unit. \end{proof}

Recall the classes $\alpha_j[n]$ that we introduced in \eqref{alpha-eq}. For every $m\geq0$ and $n\in\mathcal N$ denote by $\HH_m[n]$ the $\cO_\p[\Gal(K_m[n]/K)]$-submodule of $H^1_f(K_m[n],T)$ generated by the restrictions of the classes $z_n$ and $\alpha_j[n]$ for all $j\leq m$. Now recall the group $\hat H^1(K_\infty[n],{T})$ of \eqref{hat-H-T-eq} and define the $\Lambda[\mathcal G(n)]$-module
\[ \hat H^1_f(K_\infty[n],{T}):=\invlim_m H^1_f(K_m[n],T)\subset\hat H^1(K_\infty[n],{T})\simeq H^1(K[n],\T), \]
where the inverse limit is taken with respect to the corestriction maps. 

\begin{definition} \label{tame-dfn}
The \emph{Iwasawa module of generalized Heegner cycles of tame conductor $n$} is the compact $\Lambda[\mathcal G(n)]$-module
\[ \HH_\infty[n]:=\invlim_m\HH_m[n]\subset\hat H^1_f(K_\infty[n],{T}), \] 
the inverse limit being taken with respect to the corestriction maps. 
\end{definition}

The next result is essentially a consequence of the norm relations obtained in the lemmas above, and its proof proceeds as that of \cite[Lemma 2.3.3]{Ho1}. However, it is convenient to quickly review the arguments, as this will give us the occasion to introduce some notation that will be used later in the proof of Theorem \ref{non-triv}.

\begin{proposition} \label{prop3.6}
There exists a family  
\[ \bigl\{\beta[n]={(\beta_{m}[n])}_{m\geq0}\in\mathcal H_\infty[n]\bigr\}_{n\in\mathcal N} \]
such that $\beta_0[n]=\Phi z_n$ and
\[ \cores_{K_\infty[n\ell]/K_\infty[n]}\bigl(\beta[n\ell]\bigr)=a_\ell\cdot\beta[n] \] 
for any $n\ell\in\mathcal N$.  
\end{proposition}

\begin{proof} Fix $n\in\mathcal N$. For each $m\geq1$ let $\tilde{\mathcal H}_m$ be the free $\cO_\p[\Gal(K_m[n]/K)]$-module with a set of generators $\{x,x_j\mid0\leq j\leq m\}$ satisfying the following relations:
\begin{itemize}
\item $\sigma(x)=x$ for all $\sigma\in \Gal(K_m[n]/K[n])$; 
\item $\sigma(x_j)=x_j$ for all $\sigma\in\Gal(K_m[n]/K_j[n])$ and all $j\leq m$; 
\item $\tr_{K_j[n]/K_{j-1}[n]}(x_j)=a_px_{j-1}-p^{k-2}x_{j-2}$ for $j\geq2$; 
\item $\tr_{K_1[n]/K_0[n]}(x_1)=\gamma_1 x$ and $x_0=\gamma_0x$. 
\end{itemize}
It follows that
\begin{equation} \label{trace-prop3.6-eq}
\tr_{K_j[n]/K[n]}(x_j)=\gamma_jx
\end{equation}
for all $j\leq m$. There are inclusions $\tilde{\mathcal H}_m\hookrightarrow\tilde{\mathcal H}_{m+1}$ and canonical maps $\tr_{m}:\tilde{\mathcal H}_m\rightarrow\tilde{\mathcal H}_{m-1}$ induced by the trace. By Corollary \ref{coro5PR} and equality \eqref{trace-prop3.6-eq}, $\Phi x\in\tilde{\mathcal H}_0$ is a trace from every $\tilde{\mathcal H}_m$, so we can choose an element 
\[ y\in\tilde{\mathcal H}_\infty:=\invlim_m\tilde{\mathcal H}_m \] 
that lifts $\Phi x$. For each divisor $t$ of $n$ we define a map $\phi(t):\tilde{\mathcal H}_\infty\rightarrow\mathcal H_\infty[t]$ by sending $x_m$ to $\alpha_m[n]$ and $x$ to $z_{t}$. Now set $\beta[t]:=\phi(t)(y)$. For all $t\ell\,|\,n$ the square 
\[ \xymatrix@C=35pt@R=28pt{\tilde{\mathcal H}_\infty\ar[r]^-{\phi(t\ell)}\ar[d]^-{a_\ell}&\mathcal H_\infty[t\ell]\ar[d]^-{\cor_{K_\infty[t\ell]/K_\infty[t]}}\\\tilde{\mathcal H}_\infty\ar[r]^-{\phi(t)} & \mathcal H_\infty[t]} \]
is commutative, so we get a family ${\{\beta[t]\}}_{t\mid n}$ with the desired properties.
 
To conclude, it is enough to note that the $\Lambda$-module of these families with $t$ running 
over divisors of $n$ is compact, hence the limit over all $n\in\mathcal N$ is non-empty. \end{proof}

\subsection{Kolyvagin systems} \label{kolyvagin-classes-subsec}

Recall the classes $\beta[n]$ for $n\in\mathcal N$ constructed in Proposition \ref{prop3.6}. For each prime $\ell\in\mathcal N$ let $\sigma_\ell$ be a generator of $G_\ell$ and define 
Kolyvagin's derivative as 
\[ D_\ell:=\sum_{i=1}^\ell i\sigma_\ell^i\in\Z[G_\ell]. \]
More generally, for every $n\in\mathcal N$ set $G(n):=\prod_{\ell\mid n}G_\ell$ and $D_n:=\prod_{\ell\mid n}D_\ell$, the products being taken over all the primes $\ell$ dividing $n$. We also adopt the convention that $G(1)$ is trivial and $D_1$ is the identity operator. Now fix a set $S$ of representatives of $G(n)$ in $\mathcal G(n)$ and let 
\begin{equation} \label{derivative-eq}
\tilde\kappa_n:=\sum_{s\in S}sD_n\beta[n]\in\mathcal H_\infty[n]\subset\hat H^1(K_\infty[n],{T})\simeq H^1(K[n],\T). 
\end{equation}
The image of 
$\tilde\kappa_n$ in $H^1(K[n],\mathbf{T}/I_n\mathbf{T})$ is fixed by $\mathcal G(n)$. Define $\kappa_n\in H^1(K,\mathbf{T}/I_n\mathbf{T})$ to be the element mapping to $\tilde \kappa_n$ via the isomorphism 
\[ H^1(K,\mathbf{T}/I_n\mathbf{T})\overset\simeq\longrightarrow H^1(K[n],\mathbf{T}/I_n\mathbf{T})^{\mathcal G(n)} \]
induced by restriction. Note that this map is indeed bijective because 
\begin{itemize}
\item $H^0(K[n],\mathbf{T}/I_n\mathbf{T})\simeq\varprojlim_mH^0(K_m[n],T/I_nT)$;
\item $H^0(K_m[n],T/I_nT)=0$ since $T/I_nT\simeq A[I_n]$ and $A$ has no non-trivial $p$-torsion over $K_m[n]$ for any $m,n$.
\end{itemize}
To check the first assertion one can proceed as in the proof of \cite[Proposition II.1.1]{Colmez}, while the second claim is a consequence, by \cite[Lemma 3.10, (2)]{LV}, of the fact that the extension $K_m[n]/\Q$ is solvable.

The following result is the analogue of \cite[Lemma 2.3.4]{Ho1}. As we shall see, arguments in \emph{loc. cit.} involving reductions of elliptic curves modulo primes above $p$ will be replaced by considerations in $p$-adic Hodge theory.

\begin{lemma}
For every $n\in\mathcal N$ the class $\kappa_n$ belongs to $H^1_{\mathcal F_\Lambda(n)}(K,\mathbf{T}/I_n\mathbf{T})$. 
\end{lemma}

\begin{proof} To check that the localization of $\kappa_n$ at a prime $v\,|\,n$ lies in the transverse subspace one can follow \cite[Lemma 2.3.4]{Ho1}, which is based on the formal argument described in the proof of \cite[Lemma 1.7.3]{Ho1}.  

The case where $v\nmid Npn$ can also be treated similarly as in \cite[Lemma 2.3.4]{Ho1}, using the fact that the classes $z_{np^m}$ are unramified. 

In the case where $v\,|\,N$ one needs, as in \emph{loc. cit.}, to check that for all $v'\,|\,v$ in $K_\infty$, all $w\,|\,v$ in $K[n]$ and all $w'\,|\,w$ in $K_\infty[n]$ (note that $w$ and $v$ are finitely decomposed in the respective extensions, since all primes dividing $N$ split in $K$) the map 
\[ \bigoplus_{w'\mid w}A\bigl({K_\infty[n]}_{w'}\bigr)\longrightarrow\bigoplus_{v'\mid v}A(K_{\infty,v'}) \] 
induced by the norm is surjective: this is true because the degree of ${K_\infty[n]}_{w'}/K_{\infty,v'}$ is prime to $p$.  

The case that needs more substantial changes is the one where $v\,|\,p$, which we describe in greater detail. Let $v\,|\,p$ be a prime of $K$ and fix a prime $w$ of $K[n]$ above $v$. We still denote by $v$ and $w$ the unique primes of $K_\infty$ and $K_\infty[n]$ above $v$ and $w$, respectively, and similarly for $K_m[n]$. To simplify our notation, we set $F:=K_v$, $F_\infty:=K_{\infty,v}$, $F_m=K_{m,v}$, $F[n]:={K[n]}_w$, $F_m[n]:={K_m[n]}_w$ and $F_\infty[n]={K_\infty[n]}_w$. Define 
\[ H^1_\ord(F_m[n],\T):=\mathrm{im}\Big(H^1\bigl(F_m[n],\Fil^+_w(\T)\bigr)\longrightarrow H^1(F_m[n],\T)\Big), \]
and analogously for $F_\infty[n]$ in place of $F_m[n]$. To begin with, since $\mathcal H_m[n]\subset H^1_\ord(F_m[n],T)$ for all $m\geq0$, Shapiro's lemma shows that $\mathcal H_\infty[n]\subset H^1_\ord(F_\infty[n],\T)$. 

Now recall the filtrations $\Fil^\pm_\star(\T)$ on $\T$ (and, below, the filtrations $\Fil^\pm_\star(\mathbf{A})$ on $\mathbf{A}$) that were introduced in \S \ref{sec4.3}. To further lighten the notation, we set 
\[ \T_n:=\T/I_n\T,\qquad\T_n^+:=\Fil^+_w(\T)/I_n\Fil^+_w(\T),\qquad\T_n^-:=\Fil^-_w(\T)/I_n\Fil^-_w(\T). \] 
Then we can consider the commutative diagram
\begin{equation} \label{T-diagram-eq}
\xymatrix{H^1(F,\T_n^+)\ar[r] \ar[d]& H^1(F,\T_n)\ar[r] \ar[d]&  
H^1(F,\T_n^-)\ar[d]\\
H^1(F[n],\T_n^+)\ar[r] & H^1(F[n],\T_n)\ar[r] &  
H^1(F[n],\T_n^-)}
\end{equation}
whose vertical arrows are restrictions and whose rows are exact. The argument above shows that the image of the localization of $\kappa_n$ in $H^1(F[n],\T^-_n)$ is trivial. By the inflation-restriction sequence, the kernel of the right vertical arrow in \eqref{T-diagram-eq} is 
\[ H^1\bigl(\Gal(F[n]/F),H^0(F[n],\T_n^-)\bigr). \] 
But $H^0(F[n],\T_n^-)=0$, hence the above-mentioned map is injective. In order to justify this vanishing, observe that $\Fil_w^-(T)$ is unramified and the extension $F_{\infty}/F$ is totally ramified, hence the invariant submodule of $\T_n^-=\bigl(\Fil^-_w(T)/I_n\Fil^-_w(T)\bigr)\otimes_\cO\Lambda$ under the inertia is the inverse limit with respect to the corestriction (i.e., multiplication-by-$p$) maps of the groups $\Fil^-_w(T)/I_n\Fil^-_w(T)$. It follows that 
\begin{equation} \label{H^0-lim-eq}
H^0(F[n],\T_n^-)=\invlim_n\bigl(\Fil^-_w(T)/I_n\Fil^-_w(T)\bigr)^{\Gal(\Q_p^\mathrm{unr}/{F[n]})}.
\end{equation} 
Since $\bigl(\Fil^-_w(T)/I_n\Fil^-_w(T)\bigr)^{\Gal(\Q_p^\mathrm{unr}/{F[n]})}$ is finite, the inverse limit in \eqref{H^0-lim-eq} is zero and the claim follows. Choose a lift $\alpha\in H^1(F,\T_n^+)$ of $\kappa_n$. Since $H^0(F[n],\T_n^-)=0$, the bottom left horizontal arrow of \eqref{T-diagram-eq} is injective, hence the image of $\alpha$ in $H^1(F[n],\T_n^+)$ is the unique lift of the image of $\kappa_n$ in $H^1(F[n],\T_n)$. Set $\T^+=\Fil^+_w(\T)$. Since $\kappa_n$ belongs to the image of $H^1(F[n],\T^+)$ in $H^1(F[n],\T_n)$ by construction, the class $\alpha$ maps to zero in the right lower entry of 
\[ \xymatrix{H^1(F,\T^+)\ar[r]\ar[d] & H^1(F,\T^+_n) \ar[r]\ar[d] & H^2(F,\T^+)\ar[d]\\
H^1(F[n],\T^+)\ar[r] &  H^1(F[n],\T^+_n) \ar[r] & H^2(F[n],\T^+).} \]
To complete the proof, we only need to show that the right vertical arrow is injective. By duality, 
it is enough to show that if $\mathbf{A}^-:=\Fil^-_w(\mathbf{A})$ then the trace map
\[ \tr_{F[n]/F}: H^0(F[n], \mathbf{A}^-)\longrightarrow H^0(F,\mathbf{A}^-) \] 
is surjective. Setting again 
\[ A^-:=\Fil^-_w(A)\simeq (T^+)^*:=\Hom_{\Z_p}\bigl(T^+,F/\mathcal O(1)\bigr) \] 
and using the fact that $A^-$ is unramified, we need to check the surjectivity of 
\[ \tr_{F[n]/F}: H^0(F[n],  {A}^-)\longrightarrow H^0(F, {A}^-). \] 
The first step is to prove that $H^0(F[n],  {A}^-)$ is finite. Define $\Q_p(\zeta_{p^\infty}):=\cup_{n\geq1}\Q_p(\zeta_{p^n})$, where $\zeta_{p^n}\in\bar\Q_p$ is a primitive $p^n$-th root of unity. It is well known that the action of $G_{\Q_p(\zeta_{p^\infty})}/{I_{\Q_p(\zeta_{p^\infty})}}$ on $A^-$ is non-trivial, and this shows that $H^1(\Q_p(\zeta_{p^\infty}),A^-)$ is finite (see, e.g., \cite[p. 84]{Ochiai} for details). It follows that $H^0(\Q_p,A^-)$ is finite, and using the fact that the extension $F[n]/\Q_p$ is finite we see that $H^0(F[n],A^-)$ must be finite as well. Since the Herbrand quotient of a cyclic group acting on a finite module is $1$, it suffices to show that $H^1\bigl(F[n]/F,H^0(F[n],A^-)\bigr)=0$. This group injects into $H^1(F,A^-)$, so we are done if we prove that $H^1(F,A^-)=0$. As $A^-$ is unramified, this cohomology group is isomorphic to $A^-/(\varphi -1)A^-$ where $\varphi\in\Gal(F^\mathrm{unr}/F)$ is the Frobenius. There is a surjection $V^-\twoheadrightarrow A^-$, hence it suffices to show that $V^-/(\varphi -1)V^-=0$. As in \S \ref{sec-Bloch-Kato}, let $B_\mathrm{cris}$ be Fontaine's ring of crystalline periods, then set
\[ \D_\mathrm{cris}(M):=(M\otimes_\Q B_\mathrm{cris})^{G_F}. \] 
We know that $\D_\mathrm{cris}(V^-)^{\varphi =1}=0$ (see, e.g., \cite[p. 83]{Ochiai}), hence $\varphi -1$ is an isomorphism on $\D_\mathrm{cris}(V^-)$. Since the functor $M\mapsto\D_\mathrm{cris}(M)$ is an equivalence between the category of crystalline representations and the category of filtered admissible $\varphi$-modules, and $V^-$ is crystalline because $V$ is, it follows that $\varphi-1$ is an isomorphism on $V^-$. We conclude that $V^-/(\varphi-1)V^-=0$. \end{proof}

\begin{theorem} \label{kolyvagin-thm}
There exists a Kolyvagin system $\kappa^\mathrm{Heeg}\in\mathbf{KS}(\mathbf{T}, \mathcal F_\Lambda,\mathcal L)$ such that the class $\kappa_1^\mathrm{Heeg}\in H^1_{\mathcal F_\Lambda}(K,\mathbf{T})$ is non-zero. 
\end{theorem} 

\begin{proof} The classes $\kappa_n$, $n\in\mathcal N$, almost form a Kolyvagin system (cf. Definition \ref{def-Kol-sys}). We only need to slightly modify them in order to gain the compatibility in diagram \eqref{diagram}. We proceed as in \cite[\S 1.7]{Ho1}. For every $\ell\in \mathcal L$ let $u_\ell$ be the $p$-adic unit satisfying the relation 
\begin{equation}
\label{final-lemma}
\mathrm{loc}_\ell\left(\alpha_m[n\ell]\right)=u_\ell\cdot\phi^\mathrm{fs}_\ell\bigl(\mathrm{loc}_\ell(\alpha_m[n])\bigr).
\end{equation} 
Such a $u_\ell$ exists thanks to a combination of \cite[(K2)]{CH} and \cite[Proposition 10.2]{Nek}. If we define 
\[ \kappa'_n:=\Bigg(\prod_{\ell\mid n}u_\ell^{-1}\Bigg)\cdot \kappa_n\otimes\bigl(\otimes_{\ell\mid n}\sigma_\ell\bigr)\in H^1_{\mathcal F(n)}(K,T/I_nT)\otimes G_n \] 
then \eqref{final-lemma} ensures that $\kappa^\mathrm{Heeg}:={\{\kappa_n'\}}_{n\in \mathcal N}$ is a Kolyvagin system. Moreover, $\kappa^\mathrm{Heeg}_1=\kappa_1$. The non-triviality of $\kappa_1^\mathrm{Heeg}$ follows from Theorem \ref{non-triv} below. \end{proof}

Now we introduce the Selmer group where our Iwasawa module of generalized Heegner cycles naturally lives.

\begin{definition} \label{bloch-kato-dfn}
The \emph{pro-$p$ Bloch--Kato Selmer group of $f$ over $K_\infty$} is the $\Lambda$-module 
\[ \hat H^1_f(K_\infty,{T}):=\invlim_mH^1_f(K_m,{T}), \] 
where the inverse limit is taken with respect to the corestriction maps. 
\end{definition}

\begin{remark}
It can be shown that $\hat H^1_f(K_\infty,{T})$ is free of finite rank over $\Lambda$.
\end{remark}

For every $m\geq1$ denote by $\mathcal H_m$ the $\Lambda_m$-submodule of $H^1_f(K_m,T)$ that is generated by $\cor_{K[1]/K}(z_1)$ and $\cor_{K_m[1]/K_m}\bigl(\res_{K_m[1]/K_j[1]}(\alpha_j[1])\bigr)$ for all $j\leq m$. In line with Definition \ref{tame-dfn}, we give

\begin{definition} \label{iwasawa-module-dfn}
The \emph{Iwasawa module of generalized Heegner cycles} is the compact $\Lambda$-module
\[ \HH_\infty:=\invlim_m\HH_m\subset\hat H^1_f(K_\infty,{T}), \]
where the inverse limit is taken with respect to the corestriction maps. 
\end{definition}

\begin{remark}
If we set $K[0]:=K$ and allow for $n=0$ in Definition \ref{tame-dfn} then $\HH_\infty$ essentially coincides with $\HH_\infty[0]$. 
\end{remark}

Let $\tilde\kappa_1\in\HH_\infty[1]$ be defined as in \eqref{derivative-eq}. Since $G(1)$ is trivial and $D_1$ is the identity operator, we see that $\tilde\kappa_1=\sum_{\sigma\in\mathcal G(1)}\sigma\beta[1]$. This shows that we can view $\tilde\kappa_1$ as an element of $\mathcal H_\infty$.

\begin{theorem}\label{non-triv}
The $\Lambda$-module $\mathcal H_\infty$ is free of rank $1$, generated by $\tilde\kappa_1$.
\end{theorem} 

\begin{proof} A higher weight analogue due to Castella and Hsieh (\cite[Theorem 6.1]{CH}) of a result of Cornut (\cite{Co}) on the non-triviality of Heegner points on elliptic curves along anticyclotomic $\Z_p$-extensions ensures that $\cores_{K_m[1]/K_m}(\alpha_m[1])$ is non-torsion for $m\gg0$. Using this fact, one can show that $\mathcal H_\infty$ is free of rank $1$ over $\Lambda$ by mimicking the proof of \cite[\S 3, Proposition 10]{PR}, where an analogous result is obtained for Heegner points. Therefore it remains to show that $\mathcal H_\infty$ is generated by $\tilde\kappa_1$. We argue as in the proof of \cite[Theorem 2.3.7]{Ho1}. 

Recall that $\Gamma_m=\Gal(K_m/K)$, then write $\Gal(K_m[1]/K)\simeq\Gamma_m\times\mathcal G$ with $\mathcal G:=\mathcal G(1)$. Let $\tr_\mathcal G\in\cO_\p[\mathcal G]\subset\Lambda[\mathcal G]$ be the trace operator. Recall also the modules $\tilde\HH_m$ and $\tilde\HH_\infty$ with $n=1$ and the elements $x$, $x_j$, $y$ introduced in the proof of Proposition \ref{prop3.6}. Set, as usual, $\tilde\HH_\infty^\mathcal G:=H^0(\mathcal G,\tilde\HH_\infty)$, then define 
\[ x_j^\mathcal G:=\tr_\mathcal G(x_j)\in\tilde\HH_\infty^\mathcal G,\qquad y^\mathcal G:=\tr_\mathcal G(y)\in\tilde\HH_\infty^\mathcal G. \]
There is a commutative square
\[ \xymatrix@R=30pt@C=30pt{\tilde\HH_\infty\ar@{->>}[r]\ar@{->>}[d]^-{\tr_\mathcal G} & \mathcal H_\infty\ar@{->>}[d]^-{\cor_{H_1/K}}\\\tilde\HH_\infty^\mathcal G\ar@{->>}[r] & \mathcal H_\infty} \] 
in which all maps are surjective, the top horizontal arrow takes $x_j$ to $\alpha_j$ and the bottom horizontal arrow takes $x_j^\mathcal G$ to $\cores_{H_1/K}(\alpha_j)$ and $y^\mathcal G$ to $\tilde\kappa_1$. Fix a topological generator $\gamma$ of $\Gamma_\infty$. By Nakayama's lemma, it is enough to show that 
\begin{equation} \label{final-claim}
\text{$\tilde\HH_\infty^\mathcal G$ is isomorphic to $\Lambda y^\mathcal G+(\gamma-1)\tilde\HH_\infty^\mathcal G$.}
\end{equation} 
This is done in two steps that correspond to \cite[Lemmas 2.3.8 and 2.3.9]{Ho1}. 

For every $m\geq0$, set $\tilde\HH_m^\mathcal G:=H^0(\mathcal G,\tilde\HH_m)$. The $\cO_\p$-module $\tilde\HH_0^\mathcal G$ is free of rank $1$, generated by $x^\mathcal G$, and there is a canonical map 
\[ \Psi:\tilde\HH_\infty^\mathcal G\longrightarrow\tilde\HH_0^\mathcal G. \] 
\noindent{\bf Claim 1.} The image of $\Psi$ is a free $\cO_\p$-module of rank $1$ generated by $\Psi(y^\mathcal G)=\mathrm{aug}(\Phi)x^\mathcal G$, where $\mathrm{aug}:\cO_\p[\mathcal G]\rightarrow\cO_\p$ is the augmentation map.

\noindent{\bf Claim 2.} The map $\Psi$ induces an isomorphism 
\[ \bar\Psi:\tilde\HH_\infty^\mathcal G/(\gamma-1)\tilde\HH_\infty^\mathcal G\overset\simeq\longrightarrow \mathrm{aug}(\Phi)\tilde\HH_0^\mathcal G. \]
Clearly, Claims 1 and 2 imply \eqref{final-claim}, so it remains to justify these two assertions. 

For the first claim one can follow the proof of \cite[Lemma 2.3.8]{Ho1}, replacing \cite[\S 3.3, Corollaire 5]{PR} with our Corollary \ref{coro5PR}. More precisely, with notation as in \S \ref{PR-subsec}, one notes that $\tr_{K_m/K}(x_k^\mathcal G)=\mathrm{aug}(\gamma_m)x^\mathcal G$. Corollary \ref{coro5PR} implies that $\bigcap_{m\geq 1}\mathrm{aug}(\gamma_m)\cO_\p=\mathrm{aug}(\Phi)\cO_\p$, and then the recursive formulas in \S \ref{PR-subsec} ensure that $\mathrm{aug}(\gamma_m)\cO_\p=\mathrm{aug}(\Phi)\cO_\p$ for $m\gg0$, which implies Claim 1. 

The proof of the second claim proceeds along the lines of the proof of \cite[Lemma 2.3.9]{Ho1}, with only a minor variation in one of the recursive relations appearing there. This is due to the fact that the elements $\gamma_m$ defined in \S \ref{PR-subsec} are slightly different from their namesakes in \cite{Ho1}. However, for the reader's convenience we provide a proof, which largely overlaps the one given in \cite{Ho1}.

First we observe that we only need to show that $\bar\Psi$ is injective, since it is surjective by Claim 1. Fix $h={(h_m)}_{m\geq1}\in\ker(\Psi)$. For every $m\geq0$ the $\Lambda$-module $\tilde\HH_m^\mathcal G$ is generated by $x_m^\mathcal G$ and $x_{m-1}^\mathcal G$, so we can write 
\[ h_m=a_mx_m^\mathcal G+b_mx_{m-1}^\mathcal G+(\gamma-1)z_m \] 
for suitable $a_m,b_m\in\cO_\p$. Taking the trace to $\tilde\HH_0^\mathcal G$, and using the fact that $x^\mathcal G$ has infinite order, we obtain 
\[ a_m\mathrm{aug}(\gamma_m)+pb_m\mathrm{aug}(\gamma_{m-1})=0, \]
hence
\[ \mathrm{aug}(\gamma_m)h_m\in b_mt_m+(\gamma-1)\tilde\HH_m^\mathcal G \] 
with 
\[ t_m:=-p\,\mathrm{aug}(\gamma_{m-1})x_m^\mathcal G+\mathrm{aug}(\gamma_m)x_{m-1}^\mathcal G. \]
Applying the trace operator we get:
\begin{align*}
\tr_{K_{m+1}/K_{m}}(t_{m+1})&=-p\,\mathrm{aug}(\gamma_m)(a_p x_m^\mathcal G- p^{k-2}x_{m-1}^\mathcal G)+p\,\mathrm{aug}(a_p\gamma_{m}-p^{k-1}\gamma_{m-1})x_{m}^\mathcal G\\
&= p^{k-1}\bigl(\mathrm{aug}(\gamma_m)x_{m-1}^\mathcal G- p\,\mathrm{aug}(\gamma_{m-1})x_m^\mathcal G\bigr)\\
&=p^{k-1}t_{m}. 
\end{align*}
If we take $m$ large enough that $\mathrm{aug}(\gamma_\ell)=\mathrm{aug}(\Phi)$ for all $\ell\geq m$ then  
\[ \mathrm{aug}(\gamma_m)h_m=\mathrm{aug}(\gamma_\ell)\tr_{K_\ell/K_m}(h_\ell)\in b_\ell p^{(k-1)(\ell-m)}t_m+(\gamma-1)\tilde\HH_m^\mathcal G \] 
for $\ell\gg m$. Finally, letting $\ell\rightarrow \infty$ shows that $h_m\in(\gamma-1)\tilde\HH_m^\mathcal G$ for every $m$, and Claim 2 follows. \end{proof} 

\begin{remark}
The analogue of \cite[Theorem 6.1]{CH} when $V_{f,\p}$ is replaced by Deligne's $\ell$-adic representation with $\ell\not=p$ was proved by Howard in \cite[Theorem A]{Ho-JNT}.
\end{remark}

\begin{remark} 
For each $n\in \mathcal N$ and any integer $m\geq0$, Castella and Hsieh introduced in \cite[\S 5.2]{CH} certain \emph{$\alpha$-stabilized Heegner classes} 
\[ z^o_{f,mp^n,\alpha}\in H^1_f(K[np^m],T), \]  
where $\alpha$ is the $p$-adic unit root of the Hecke polynomial $X^2-a_pX+p^{k-1}$. These classes are defined in terms of the elements $z_{mp^n}$ of \S \ref{generalized-heegner-subsec} via a regularization process that is analogous to the one used in \cite[\S 2.5]{BD96} in the case of Heegner points on elliptic curves. As is shown in \cite[Lemma 5.3]{CH}, for each $n$ one has $\cores_{K[np^m]/K[np^{m-1}]}(z^o_{f,np^m,\alpha})=\alpha z^o_{f,np^{m-1},\alpha}$, so we get an element 
\[ \mathbf{z}_{f,n,\alpha}:=\invlim_m\alpha^{-n}z^o_{f,np^m,\alpha}\in\invlim_mH^1_f(K[np^m],T). \]
Using the elements $\mathbf{z}_{f,n,\alpha}$, one might be able to show that the $\Lambda$-module $\mathcal H_\infty$ is free of rank $1$ by adapting the techniques developed by Bertolini in \cite{Ber1}.
\end{remark}

\section{Proof of Theorem \ref{main-thm} and Main Conjecture}

In this final section we prove the main result of this paper (Theorem \ref{main-thm}) and formulate a Main Conjecture \emph{\`a la} Perrin-Riou for generalized Heegner cycles, one divisibility of which is the content of part (3) of Theorem \ref{main-thm}. 

\subsection{Proof of Theorem \ref{main-thm}} \label{main-proof-subsec}

We note that hypotheses (H.0)--(H.5) and Assumption \ref{ass1} in \S \ref{triples-subsec} are satisfied in this setting. First of all, the ordinariness of $f$ that we imposed in part (5) of Assumption \ref{ass} ensures that condition (1) in Assumption \ref{ass1} is satisfied. Also, in this case, for any number field $E$ there is a natural identification 
\[ \Sel_{\mathrm{Gr}}(T/E)=H^1_f(E,T) \]
between Greenberg's and Bloch--Kato's Selmer groups (cf. \cite[\S 2.4]{Howard-Inv}). Moreover, by comparing local conditions, one can check that there is a canonical isomorphism
\[ \hat H^1_f(K_\infty,T)\simeq H^1_{\mathcal F_\Lambda}(K,\mathbf{T}), \]
where $\mathcal F_\Lambda$ is the Selmer structure introduced in \S \ref{sec4.3}.

The self-duality of $V$ implies that, possibly after rescaling the isomorphism 
\[ V\simeq V^*:=\Hom_{F_\p}\bigl(V,F_\p(1)\bigr), \] 
our chosen lattice $T\subset V$ is also self-dual, i.e., $T\simeq T^*:=\Hom_{\cO_\p}\bigl(T,\mathcal O_\p(1)\bigr)$. This proves the existence of the pairing in (H.4). It is well known that $\Fil_v^+(T)$ is maximal orthogonal with respect to this pairing, so condition (2) in Assumption \ref{ass1} is satisfied. On the other hand, condition (3) in Assumption \ref{ass1} is \cite[p. 84]{Ochiai}. Finally, hypothesis (H.5) can be checked as in the proof of \cite[Proposition 2.1.3]{Ho1} and then Theorem \ref{main-thm} follows by combining Theorems \ref{MR} and \ref{kolyvagin-thm}. 

\subsection{Main Conjecture} 

Recall the finitely generated torsion $\Lambda$-module $M$ in part (2) of Theorem \ref{main-thm}.  

\begin{conjecture}[Main Conjecture] \label{main-conj}
$\mathrm{char}(M)=\mathrm{char}\Big(\hat H^1_{f}(K_\infty,T)/\mathcal H_\infty\Big)$.
\end{conjecture}

As remarked above, part (3) of Theorem \ref{main-thm} shows that one divisibility in Conjecture \ref{main-conj} holds true. 

Conjecture \ref{main-conj} should be compared with the ``main conjecture'' for Heegner points on elliptic curves that was proposed by Perrin-Riou in \cite{PR}. Important partial results towards Perrin-Riou's conjecture were obtained by Bertolini (\cite{Ber1}) and Howard (\cite{Ho1}, \cite{Ho2}). Recently, a proof of Perrin-Riou's conjecture has been announced by Wan (\cite{Wan}).

\bibliographystyle{amsplain}
\bibliography{Iwasawa}

\providecommand{\bysame}{\leavevmode\hbox to3em{\hrulefill}\thinspace}
\providecommand{\MR}{\relax\ifhmode\unskip\space\fi MR }
\providecommand{\MRhref}[2]{%
  \href{http://www.ams.org/mathscinet-getitem?mr=#1}{#2}
}
\providecommand{\href}[2]{#2}
\begin{thebibliography}{10}

\bibitem{Ber1}
M.~Bertolini, \emph{Selmer groups and {H}eegner points in anticyclotomic
  {$\bold Z_p$}-extensions}, Compos. Math. \textbf{99} (1995), no.~2, 153--182.

\bibitem{Ber2}
\bysame, \emph{Iwasawa theory for elliptic curves over imaginary quadratic
  fields}, J. Th\'eor. Nombres Bordeaux \textbf{13} (2001), no.~1, 1--25.

\bibitem{BD96}
M.~Bertolini and H.~Darmon, \emph{Heegner points on {M}umford--{T}ate curves},
  Invent. Math. \textbf{126} (1996), no.~3, 413--456.

\bibitem{BDP}
M.~Bertolini, H.~Darmon, and K.~Prasanna, \emph{Generalized {H}eegner cycles
  and {$p$}-adic {R}ankin {$L$}-series}, Duke Math. J. \textbf{162} (2013),
  no.~6, 1033--1148.

\bibitem{BK}
S.~Bloch and K.~Kato, \emph{{$L$}-functions and {T}amagawa numbers of motives},
  The {G}rothendieck {F}estschrift, {V}ol.\ {I}, Progr. Math., vol.~86,
  Birkh\"auser Boston, Boston, MA, 1990, pp.~333--400.

\bibitem{Bru}
A.~Brumer, \emph{Pseudocompact algebras, profinite groups and class
  formations}, J. Algebra \textbf{4} (1966), no.~3, 442--470.

\bibitem{CH}
F.~Castella and M.-L. Hsieh, \emph{Heegner cycles and $p$-adic
  ${L}$-functions}, submitted (2015).

\bibitem{Colmez}
P.~Colmez, \emph{Th\'eorie d'{I}wasawa des repr\'esentations de de {R}ham d'un
  corps local}, Ann. of Math. (2) \textbf{148} (1998), no.~2, 485--571.

\bibitem{Co}
C.~Cornut, \emph{Mazur's conjecture on higher {H}eegner points}, Invent. Math.
  \textbf{148} (2002), no.~3, 495--523.

\bibitem{Del-Bourbaki}
P.~Deligne, \emph{Formes modulaires et repr\'esentations {$l$}-adiques},
  S\'eminaire {B}ourbaki. {V}ol. 1968/69: {E}xpos\'es 347--363, Lecture Notes
  in Math., vol. 175, Springer, Berlin, 1971, pp.~139--172.

\bibitem{Fouquet}
O.~Fouquet, \emph{Dihedral {I}wasawa theory of nearly ordinary quaternionic
  automorphic forms}, Compos. Math. \textbf{149} (2013), no.~3, 356--416.

\bibitem{Ho1}
B.~Howard, \emph{The {H}eegner point {K}olyvagin system}, Compos. Math.
  \textbf{140} (2004), no.~6, 1439--1472.

\bibitem{Ho2}
\bysame, \emph{Iwasawa theory of {H}eegner points on abelian varieties of {$\rm
  GL_2$} type}, Duke Math. J. \textbf{124} (2004), no.~1, 1--45.

\bibitem{Ho-JNT}
\bysame, \emph{Special cohomology classes for modular {G}alois
  representations}, J. Number Theory \textbf{117} (2006), no.~2, 406--438.

\bibitem{Howard-Inv}
\bysame, \emph{Variation of {H}eegner points in {H}ida families}, Invent. Math.
  \textbf{167} (2007), no.~1, 91--128.

\bibitem{Kol-Euler}
V.~A. Kolyvagin, \emph{Euler systems}, The {G}rothendieck {F}estschrift,
  {V}ol.\ {II}, Progr. Math., vol.~87, Birkh\"auser Boston, Boston, MA, 1990,
  pp.~435--483.

\bibitem{LT}
S.~Lang and H.~Trotter, \emph{Frobenius distributions in {${\rm
  GL}_{2}$}-extensions}, Lecture Notes in Mathematics, Vol. 504,
  Springer-Verlag, Berlin-New York, 1976.

\bibitem{LV}
M.~Longo and S.~Vigni, \emph{A refined {B}eilinson--{B}loch {C}onjecture for
  motives of modular forms}, Trans. Amer. Math. Soc., to appear.

\bibitem{Maz2}
B.~Mazur, \emph{Modular curves and arithmetic}, Proceedings of the
  {I}nternational {C}ongress of {M}athematicians, {V}ol.\ 1, 2 ({W}arsaw,
  1983), PWN, Warsaw, 1984, pp.~185--211.

\bibitem{MR}
B.~Mazur and K.~Rubin, \emph{Kolyvagin systems}, Mem. Amer. Math. Soc.
  \textbf{168} (2004), no.~799.

\bibitem{MM}
M.~R. Murty and V.~K. Murty, \emph{A variant of the {L}ang--{T}rotter
  conjecture}, Number theory, analysis and geometry, Springer, New York, 2012,
  pp.~461--474.

\bibitem{MMS}
M.~R. Murty, V.~K. Murty, and N.~Saradha, \emph{Modular forms and the
  {C}hebotarev density theorem}, Amer. J. Math. \textbf{110} (1988), no.~2,
  253--281.

\bibitem{Nek}
J.~Nekov{\'a}{\v{r}}, \emph{Kolyvagin's method for {C}how groups of
  {K}uga--{S}ato varieties}, Invent. Math. \textbf{107} (1992), no.~1, 99--125.

\bibitem{NSW}
J.~Neukirch, A.~Schmidt, and K.~Wingberg, \emph{Cohomology of number fields},
  Grundlehren der Mathematischen Wissenschaften, vol. 323, Springer-Verlag,
  Berlin, 2000.

\bibitem{Ochiai}
T.~Ochiai, \emph{Control theorem for {B}loch--{K}ato's {S}elmer groups of
  {$p$}-adic representations}, J. Number Theory \textbf{82} (2000), no.~1,
  69--90.

\bibitem{PR}
B.~Perrin-Riou, \emph{Fonctions {$L$} {$p$}-adiques, th\'eorie d'{I}wasawa et
  points de {H}eegner}, Bull. Soc. Math. France \textbf{115} (1987), no.~4,
  399--456.

\bibitem{Serre-Cheb}
J.-P. Serre, \emph{Quelques applications du th\'eor\`eme de densit\'e de
  {C}hebotarev}, Inst. Hautes \'Etudes Sci. Publ. Math. \textbf{54} (1981),
  323--401.

\bibitem{Wan}
X.~Wan, \emph{Heegner point {K}olyvagin system and {I}wasawa {M}ain
  {C}onjecture}, preprint (2014).

\end{thebibliography}
\end{document}